\documentclass{amsart}
\usepackage{amsmath,amsfonts,amssymb,amsthm,epsfig,amscd,epsfig,psfrag,comment}
\usepackage[all]{xy}

\newtheorem{Theorem}{Theorem}[section]
\newtheorem{Proposition}[Theorem]{Proposition} 
\newtheorem{Lemma}[Theorem]{Lemma}

\newtheorem{Corollary}[Theorem]{Corollary}

\theoremstyle{definition}

\newtheorem{Remark}[Theorem]{Remark}

\newcommand{\p}{{\mathbb{P}}}
\newcommand{\Z}{\mathbb{Z}}
\newcommand{\C}{\mathbb{C}}

\def\bA{{\mathbb{A}}}
\def\k{{\mathbb{C}}}
\def\id{{\mathrm{Id}}}

\def\sl{{\mathfrak{sl}}}
\def\g{{\mathfrak{g}}}

\def\l{{\lambda}}
\def\E{{\sf{E}}}
\def\F{{\sf{F}}}

\def\sF{{\mathcal{F}}}
\def\sK{{\mathcal{K}}}

\def\G{{\mathbb{G}}}

\def\T{{\sf{T}}}

\def\tY{{\tilde{Y}}}

\def\tB{{\tilde{B}}}

\def\sL{{\mathcal{L}}}
\def\sP{{\mathcal{P}}}

\def\sQ{{\mathcal{Q}}}

\newcommand{\Cone}{\mathrm{Cone}}
\def\O{{\mathcal O}}
\def\sE{{\mathcal{E}}}
\def\H{{\mathcal{H}}}

\def\dim{\mbox{dim}}
\def\codim{\mbox{codim}}

\newcommand{\base}{\mathbb{C}[q,q^{-1}]}

\newcommand{\vect}{\C((z))^m}
\newcommand{\Gr}{\mathrm{Gr}}

\DeclareMathOperator{\spn}{span}
\DeclareMathOperator{\Hom}{Hom}

\DeclareMathOperator{\supp}{supp}
\DeclareMathOperator{\End}{End}

\DeclareMathOperator{\flip}{Flip}

\begin{document}

\title{Categorical Geometric skew Howe duality}

\author{Sabin Cautis}
\email{scautis@math.harvard.edu}
\address{Department of Mathematics\\ Rice University \\ Houston, TX}

\author{Joel Kamnitzer}
\email{jkamnitz@math.toronto.edu}
\address{Department of Mathematics\\ University of Toronto \\ Toronto, ON}

\author{Anthony Licata}
\email{amlicata@math.stanford.edu}
\address{Department of Mathematics\\ Stanford University \\ Palo Alto, CA}

\begin{abstract}
We categorify the R-matrix isomorphism between tensor products of minuscule representations of 
$U_q( \sl_n) $ by constructing an equivalence between the derived categories of coherent sheaves on the corresponding convolution products in the affine Grassmannian.  The main step in the construction is a categorification of representations of $U_q( \sl_2)$ which are related to representations of $U_q(\sl_n)$ by quantum skew Howe duality.  The resulting equivalence is part of the program of algebro-geometric categorification of Reshitikhin-Turaev tangle invariants developed by the first two authors.
\end{abstract}

\date{\today}
\maketitle
\tableofcontents

\section{Introduction}

This paper is part of the program laid out in \cite{ck2} for categorifying Reshitikhin-Turaev tangle invariants  using geometric representation theory.  Let us briefly recall the main idea explained in the introduction to that paper.  Let $G$ be a complex reductive algebraic group.  If $ (\lambda_1, \dots, \lambda_n) $ is a sequence of dominant minuscule weights of $ G $, then there is a smooth projective variety $ \Gr_{\lambda_1} \tilde{\times} \cdots \tilde{\times} \Gr_{\lambda_n} $ (constructed using the affine Grassmannian of the Langlands dual group) whose cohomology is naturally isomorphic to $ V_{\lambda_1} \otimes \cdots \otimes V_{\lambda_n} $ (see \cite{MVi}).  The Reshitikhin-Turaev tangle invariants are maps between various $ V_{\lambda_1} \otimes \cdots \otimes V_{\lambda_n} $ (or their quantum analogs).  We proposed in \cite{ck2} to lift these linear maps to functors between the derived categories of coherent sheaves on the varieties $ \Gr_{\lambda_1} \tilde{\times} \cdots \tilde{\times} \Gr_{\lambda_n} $. 

A first step in this direction is to define a functor corresponding to a crossing on two strands.  This will be an equivalence of derived categories
\begin{equation*}
D(\Gr_{\lambda} \tilde{\times} \Gr_\mu) \rightarrow D(\Gr_\mu \tilde{\times} \Gr_\lambda)
\end{equation*}
which lifts the isomorphism $ V_\lambda \otimes V_\mu \rightarrow V_\mu \otimes V_\lambda $ of Grothendieck groups.  The main purpose of this paper is to construct such an equivalence when $ \lambda, \mu $ are minuscule representation of $ SL_m $. We will use the notation $ Y(k,l) := \Gr_{\omega_k} \tilde{\times} \Gr_{\omega_l}$. A concrete definition of $ Y(k,l) $ is given in section \ref{se:varieties}.

Directly constructing such an equivalence turns out to be surprisingly difficult, even though there is a natural correspondence $ Z(k,l) $ between the two varieties $ Y(k,l) $ and $ Y(l,k)$.  As one indication of difficulty, consider the case $k=l=2$, $ m=4$; in this case there is an open subset of $ Y(2,2) $ isomorphic to the cotangent bundle of the Grassmannian of two-dimensional subspaces in $ \C^4$, $ T^* \G(2,4) $, and it was shown by Namikawa \cite{nam2} that the natural correspondence between $ T^* \G(2,4) $ and itself does not give an equivalence of derived categories. Hence in order to construct the equivalence, we use a more indirect approach involving categorical $\sl_2$ actions.  To explain how this $\sl_2$ action enters the picture, we must first describe geometric skew Howe duality.

\subsection{Geometric skew Howe duality}

Consider the vector space $ \Lambda^N (\C^2 \otimes \C^m)$.  This vector space has commuting actions of 
both $ SL_2 $ and $ SL_m $, each of which is the others commutant.  This fact is known as skew Howe duality.

The skew Howe duality gives us a surjection
\begin{equation*}
U \sl_2 \twoheadrightarrow \End_{SL_m} \big(\Lambda^N (\C^2 \otimes \C^m) \big).
\end{equation*}
We can decompose $ \Lambda^N(\C^2 \otimes \C^m) $ into weight spaces for the maximal torus in $ SL_2 $ as
\begin{equation*}
\Lambda^N(\C^2 \otimes \C^m) = \bigoplus_{k + l = N} \Lambda^k \C^m \otimes \Lambda^l \C^m.
\end{equation*}
It is easy to see that (up to sign), the action of $  t = \left( \begin{smallmatrix} 0 & 1 \\ -1 & 0 \end{smallmatrix} \right) \in SL_2 $ on $ \Lambda^N(\C^2 \otimes \C^m) $ permutes the tensor factors  on the right hand side of this equation.  

A geometric setup related to this skew Howe duality was described by Mirkovic-Vybornov \cite{MVy}.  Under the geometric Satake correspondence, the $SL_2$ representation $ \Lambda^N (\C^2 \otimes \C^m) $ corresponds to the perverse sheaf
\begin{equation*}
\bigoplus_{k+l = N} \pi_* \C_{Y(k,l)}
\end{equation*}
on a piece $\Gr$ of the affine Grassmannian. Here $\C_{Y(k,l)}$ denotes the constant sheaf on $Y(k,l)$ and $ \pi $ denotes the map $ Y(k,l) \rightarrow \Gr $ (described in section \ref{se:varieties}).  Since the geometric Satake correspondence is an equivalence of categories, endomorphisms of $\Lambda^N(\C^2 \otimes \C^m)$ go to endomorphisms of $\bigoplus_{k+l=N} \pi_* \C_{Y(k,l)}$ and we get a surjection
\begin{equation*}
U \sl_2 \twoheadrightarrow \End \big( \bigoplus_{k+l = N} \pi_* \C_{Y(k,l)}) = \bigoplus_{k+l = N = k'+l'} H_{top} (Y(k,l) \times_{\Gr} Y(k',l') \big).
\end{equation*}
(The second equality follows from the fact that $ \pi $ is semi-small -- see \cite[Theorem 8.5.7]{CG}.)

It follows that $\sl_2$ acts on the cohomology $\oplus_{k+l=N} H^*(Y(k,l))$ , with the action of the generators $E, F$ coming from the correspondences $ W(k,l) $ described in section \ref{se:functors}.  Mirkovic-Vybornov \cite{MVy} explained that when restricted to certain transverse slices, this action of $ \sl_2 $ becomes the action studied by Ginzburg \cite[Chapter 4]{CG} and Nakajima \cite{Nak} (Mirkovic-Vybornov also considered a more general setup where $ \C^2 $ and $ \sl_2 $ are replaced by $ \C^n $ and $\sl_n$).

\subsection{Categorical $ \sl_2 $ actions and construction of the equivalence}

The additional structure of an action of $ \sl_2 $ on $ \oplus H^*(Y(k,l)) $  suggests the following strategy for constructing our desired equivalence, inspired by the work of Chuang-Rouquier \cite{CR}.  First, construct an action of $ \sl_2 $ on $ \oplus_{k,l} D(Y(k,l)) $. Then use this action to construct an action of the reflection element $ t \in SL_2 $ on $ \oplus_{k,l} D(Y(k,l)) $. This gives us an equivalence $\T : D(Y(k,l)) \xrightarrow{\sim} D(Y(l,k))$ lifting the isomorphism $\Lambda^k \C^m \otimes \Lambda^l \C^m \xrightarrow{\sim} \Lambda^l \C^m \otimes \Lambda^k \C^m$ of $ \sl_2 $ weight spaces. 

More precisely, we construct a geometric categorical $\sl_2$ action using the varieties $Y(k,l)$ (see Theorem \ref{thm:main}). We introduced the notion of a geometric categorical $\sl_2$ action in \cite{ckl1} (but review it in section \ref{se:geomsl2def} of this paper). Roughly speaking, this means that we have functors $ \E, \F$ between the categories $ D(Y(k,l)) $ which obey the relations in the Lie algebra $ \sl_2 $ at the level of cohomology of complexes.  An additional part of the data of a geometric categorical action are compatible deformations $\tY(k,l)$ of the varieties $Y(k,l)$; these deformations play a crucial role in lifting the $ \sl_2 $ relations from isomorphisms at the level of cohomology to genuine isomorphisms of functors.

Using the main result of \cite{ckl1}, this geometric categorical $\sl_2$ action induces a strong categorical $\sl_2$ action on the categories $ D(Y(k,l))$.  In \cite{ckl2} (following ideas of Chuang-Rouquier \cite{CR}), we proved that a strong categorical $\sl_2$ action allows one to lift the reflection element $ t \in SL_2 $ to obtain equivalences.  Applying this to our situation we obtain the desired equivalence $ \T : D(Y(k,l)) \xrightarrow{\sim} D(Y(l,k)) $ (see Theorem \ref{thm:equivalence}).  

We work throughout with $\C^\times$-equivariant derived categories.  In particular the Grothedieck groups are $ \base $ modules and we construct an isomorphism of $ \C[q,q^{-1}] $-modules $ K(D(Y(k,l))) \rightarrow \Lambda_q^k(\C^m) \otimes \Lambda_q^l(\C^m) $ where $ \Lambda_q^k(\C^m) $ denotes the $ U_q(\sl_m) $ module corresponding to $ \Lambda^k(\C^m)$.  

The skew Howe duality picture described above extends to the quantum setting.  Using a modification of work of Toledano Laredo on the quantum Howe duality, we show that on the level of Grothendieck groups, our equivalence $ \T $ induces the braiding isomorphism $ \Lambda_q^k(\C^m) \otimes \Lambda_q^l(\C^m) \rightarrow \Lambda_q^l(\C^m) \otimes \Lambda_q^k(\C^m) $ up to multiplication by scalars (Theorem \ref{thm:equivalence}).  In the case when $ k, l \in \{ 1, m -1 \} $, this gives further evidence for Conjecture 7.1 from \cite{ck2} -- with some care it should lead to a proof of that conjecture.

\subsection{Categorical $\g$ actions} 

This paper also serves as the first step in constructing categorical $\g$ actions in algebro-geometric settings where $\g$ is an arbitrary Kac-Moody Lie algebra. In \cite{ckl3} we will construct a geometric categorical $\g$ action on quiver varieties associated to the Dynkin diagram of $\g$. This lifts Nakajima's analogous action on the cohomology of these quiver varieties and can be used to construct braid group actions. 

The proof in \cite{ckl3} builds on the work in this paper (namely Theorem \ref{thm:main}). To prove that we have a categorical $\g$ action we reduce to proving that we have compatible $\sl_2$ actions (one action for each root in $\g$). This in turn follows from the computations in this paper (either using a formal argument or by very similar calculations). 

\subsection{Organization}
The paper is organized as follows. In section \ref{sec:setup} we describe our varieties and the functors which define the $\sl_2$ action. In section \ref{sec:mainresult} we define the notion of geometric categorical $ \sl_2 $ action and explain how we achieve such an action in our case.  In section \ref{sec:equiv}, we explain the equivalence we construct and the map it induces on the level of Grothendieck groups.  In section \ref{sec:proofs}, we prove that we have a geometric categorical $ \sl_2 $ action using some  technical results and calculations from section \ref{sec:prems}. For a quick overview of the paper one should read sections \ref{sec:setup}, \ref{sec:mainresult} and \ref{sec:equiv}. 
 
\subsection{Acknowledgements}
We would like to thank Arkady Berenstein, Alexander Braverman, Dennis Gaitsgory, Valerio Toledano Laredo, and Sebastian Zwicknagl for helpful comments. S.C. was supported by National Science Foundation Grant 0801939 and thanks the Mathematical Sciences Research Institute for its support and hospitality. J.K. was supported by a fellowship from the American Institute of Mathematics. A.L. would also like to thank the Max Planck Institute in Bonn for hospitality and support.

\section{Geometric setup}\label{sec:setup}

In this section we begin by reviewing the notation of Fourier-Mukai kernels and Grassmanians.  We then
introduce the varieties $Y(k,l)$, the deformations $\tY(k,l)$ and the $\sl_2$ functors $\E$ and $\F$.

\subsection{Notation}

For a smooth variety $X$, let $D(X)$ denotes the bounded derived category of $\C^\times$ equivariant coherent sheaves. An object in $\sP \in D(X)$ is represented by a complex of quasi-coherent sheaves. The cohomology of this complex, which we denote $\H^*(\sP)$, are coherent sheaves on $X$. 

If $Y$ is another smooth variety with $\C^\times$ action then an object $\sP \in D(X \times Y)$ whose support is proper over $Y$ induces a Fourier-Mukai (FM) functor $\Phi_{\sP}: D(X) \rightarrow D(Y)$ via $(\cdot) \mapsto \pi_{2*}(\pi_1^* (\cdot) \otimes \sP)$ (where every operation is derived). One says that $\sP$ is the FM kernel which induces $\Phi_{\sP}$. The right and left adjoints $\Phi_{\sP}^R$ and $\Phi_{\sP}^L$ are induced by $\sP_R := \sP^\vee \otimes \pi_2^* \omega_X [\dim(X)]$ and $\sP_L := \sP^\vee \otimes \pi_1^* \omega_Y [\dim(Y)]$ respectively.  If $\sQ \in D(Y \times Z)$ then $\Phi_{\sQ} \circ \Phi_{\sP} \cong \Phi_{\sQ * \sP}: D(X) \rightarrow D(Y)$ where $\sQ * \sP = \pi_{13*}(\pi_{12}^* \sP \otimes \pi_{23}^* \sQ)$ is the convolution product (see also \cite{ck1} section 3.1). 

Let $\G(k,n)$ denote the Grassmannian of $k$-planes in $\C^n$.  We will denote by $H^\star(\G(k,n))$ its  bigraded symmetric cohomology. This means that we first shift the usual grading on the cohomology of $\G(k,n)$ so that it is symmetric with respect to degree zero and then add a second grading $\{\cdot\}$ so that the resulting bigraded vector space is supported on the anti-diagonal; thus every homogeneous subspace of $H^\star(\G(k,n))$ has been given a shift of the form $[s]\{-s\}$ for some $s \in \Z$. For example, we have
$$H^\star(\p^n) \cong \C[n]\{-n\} \oplus \C[n-2]\{-n+2\} \oplus \dots \oplus \C[-n+2]\{n-2\} \oplus \C[-n]\{n\}$$
as a bigraded vector space.

For any variety $ X $, we will use $ \Delta $ to denote the diagonal copy of $ X $ in $ X \times X $.

\subsection{The varieties} \label{se:varieties}

Fix positive integers $ m,N$.  Let $ k, l \le m $ with $ k+l = N$.  We define varieties $ Y(k,l) $ by
\begin{equation*}
Y(k,l) := \{ \C[[z]]^m = L_0 \subset L_1 \subset L_2 \subset \vect : z L_i \subset L_{i-1}, \dim(L_2/L_1) = l, \dim(L_1/L_0) = k \}
\end{equation*}
where the $L_i $ are $ \C$ vector subspaces.

We will say that $ Y(k,l) $ is a ``weight space'' with weight $ \l = l-k $. Note that forgetting $L_2$ describes 
$Y(k,l)$ as a $\G(l,m)$ bundle over $\G(k,m)$. 

There are two natural vector bundles on $Y(k,l)$ whose fibres over a point $\{L_0 \subset L_1 \subset L_2\}$ are $L_1$ and $L_2$, respectively. Abusing notation we will denote these bundles by $L_1$ and $L_2$. 

To explain the relation of $Y(k,l)$ to the affine Grassmannian, let $ \Gr $ denotes the variety 
\begin{equation*}
\Gr := \{\C[[z]] = L_0  \subset L \subset \vect : z L \subset L\}
\end{equation*}
$\Gr$ is a piece of the affine Grassmannian for $ GL_m$ (to get the full thing we would not require $ L_0 \subset L $).   For a dominant integral positive weight $\mu = (\mu_1 \ge \dots \ge \mu_m \ge 0) $ of $GL_m$, let $\Gr_\mu \subset \Gr$ denote the locus where $z$ has Jordan type 
$\mu$ on the quotient $L/L_0$. 

Then there is a natural map $ \pi : Y(k,l) \rightarrow \Gr $, taking $ L_0 \subset L_1 \subset L_2 $ to $ L_0 \subset L_2 $. The image of the map $ Y(k,l) \rightarrow \Gr $ is $ \overline{\Gr_{\omega_k + \omega_l}} $ where $\omega_1, \dots, \omega_{m-1}$ are the fundamental weights of $GL_m$. The map $ Y(k,l) \rightarrow \overline{\Gr_{\omega_k + \omega_l}}$ is exactly semi-small.  Similarly, there exists a map $Y(l,k) \rightarrow \overline{\Gr_{\omega_l + \omega_k}}$, so that $Y(k,l)$ and $Y(l,k)$ are birational and related by contractions to their common image $\overline{\Gr_{\omega_l + \omega_k}}$.  

In fact $ Y(k,l) $ is the convolution product of $ \Gr_{\omega_k} $ with $ \Gr_{\omega_l} $ (see Proposition 2.1 of \cite{ck2}).

We define correspodences $W^r(k,l) \subset Y(k,l) \times Y(k+r, l-r) $ by
\begin{equation*}
W^r(k,l) := \{ (L_\bullet, L'_\bullet) : L_1 \subset L'_1, L_2 = L'_2 \}.
\end{equation*}
Hence a point in $ W^r(k,l) $ is a flag $ L_0 \subset L_1 \subset L'_1 \subset L_2 = L'_2 $ with jumps of size $k,r,l-r$ such that $zL_2 \subset L_1$ and $z L_1' \subset L_0$. 
To keep track of the various jump sizes we add the integers $k,r,l-r$ as superscripts above the corresponding inclusions.  In this notation, the correspondence $W^r(k,l)$ would be described as follows
$$W^r(k,l) = \{L_0 \xrightarrow{k} L_1 \xrightarrow{r} L_1' \xrightarrow{l-r} L_2: zL_1' \subset L_0 \text{ and } zL_2 \subset L_1 \}.$$ 
There are natural vector bundles $L_1, L_1'$ and $L_2$ on $W^r(k,l)$. 

\subsection{The $\C^\times$ action}

There is an action of $\C^\times$ on $\C((z))$ given by $t \cdot z^k = t^{2k}z^k$. This induces an action of $\C^\times$ on $\vect$. Notice that for any $v \in \vect$ we have 
$$t \cdot (zv) = t^2 z(t \cdot v).$$
Thus, as vector spaces, $t \cdot (zL_i) = z(t \cdot L_i)$, so if $zL_i \subset L_{i-1}$ then $t \cdot zL_i \subset t \cdot L_{i-1}$; thus $z(t \cdot L_i) \subset t \cdot L_{i-1}$. Consequently, the $\C^\times$ action on $\vect$ induces a $\C^\times$ action on $Y(k,l)$:
$$t \cdot (L_0,L_1,L_2) = (t \cdot L_0, t \cdot L_1, t \cdot L_2).$$  

We will always work $\C^\times$-equivariantly with respect to this action. So, for example, in this paper 
$D(Y(k,l))$ will denote the bounded derived category of $\C^\times$ equivariant coherent sheaves (as opposed to the usual $D^b_{\C^\times}(Y(k,l))$). The subvariety 
$W^r(k,l) \subset Y(k,l) \times Y(k+r,l-r)$ and the vector bundles $L_i$ defined earlier are also 
$\C^\times$-equivariant. 

If $Y$ carries a $\C^\times$ action we denote by $\O_Y\{k\}$ the structure sheaf of $Y$ with non-trivial $\C^\times$ action of weight $k$. More precisely, if $f \in \O_Y(U)$ is a local function then, viewed as a section $f' \in \O_Y\{k\}(U)$, we have $t \cdot f' = t^{-k}(t \cdot f)$. If $\mathcal{M}$ is a $\C^\times$-equivariant sheaf then we define $\mathcal{M}\{k\} := \mathcal{M} \otimes \O_Y \{k\}$. Using this notation we have the $\C^\times$-equivariant map $z: L_{i+1} \rightarrow L_i \{2\}$. 

\subsection{The deformations}

Each variety $Y(k,l)$ has a natural 2-parameter deformation over $\bA_\C^2 = \langle x,y \rangle$ given by
\begin{eqnarray*}
\{ \C[[z]]^m = L_0 \subset L_1 \subset L_2 \subset \vect; (x,y) \in \C^2 : \\
(z-x) L_1 \subset L_0, (z-y) L_2 \subset L_1, \dim(L_2/L_1) = l, \dim(L_1/L_0) = k \}.
\end{eqnarray*}
Notice that over $(x,y)=(0,0)$ we recover $Y(k,l)$. This deformation restricted to the diagonal $x=y$ is actually trivial but if we take any other ray in $\bA_\C^2$ through the origin we get a non-trivial deformation of $Y(k,l)$. Which ray we choose will not be important -- we choose the axis $y=0$ to obtain the deformation
\begin{eqnarray*}
\tY(k,l) := \{\C[[z]]^m = L_0 \subset L_1 \subset L_2 \subset \vect; x \in \C : \\
(z-x) L_1 \subset L_0, z L_2 \subset L_1, \dim(L_2/L_1) = l, \dim(L_1/L_0) = k \}.
\end{eqnarray*}
The $\C^\times$ action on $Y(k,l)$ extends to a $\C^\times$ action on all of $\tY(k,l)$.  This action maps fibers to fibers and acts on the base $\bA^1_\C$ via $x \mapsto t^2 x$. 

\subsection{The functors} \label{se:functors}

For $r \geq 0$ we define functors $ \E^{(r)}(k,l) : D(Y(k+r, l-r)) \rightarrow D(Y(k,l)) $ as the Fourier-Mukai (FM) transform with respect to the sheaf
\begin{equation*}
\sE^{(r)}(k,l) := \O_{W^r(k,l)} \otimes \det(L_2/L'_1)^{-r} \det(L_1/L_0)^r \{rk\}
\end{equation*}

Similarly, 
define functors $ \F^{(r)}(k,l) : D(Y(k,l)) \rightarrow D(Y(k+r, l-r)) $ as the FM transform with respect to the sheaf
\begin{equation*}
\sF^{(r)}(k,l) := \O_{W^r(k,l)} \otimes \det(L'_1/L_1)^{l-k-r} \{r(l-r)\}.
\end{equation*}

These $\E$s and $\F$s are naturally left and right adjoints of each other up to shifts (see section \ref{sec:adjunctions}). 

\section{Construction of the geometric categorical $ \sl_2 $ action}\label{sec:mainresult}

In this section we give the definition of a geometric categorical $\sl_2$ action.  Then we explain how the varieties and functors introduced in section \ref{sec:setup} induce a geometric categorical $\sl_2$ action.  This allows us to use the main results from \cite{ckl1,ckl2} to give us equivalences $ D(Y(k,l)) \rightarrow D(Y(l,k)) $ (section \ref{sec:equiv}).

\subsection{Geometric categorical $\sl_2$ action} \label{se:geomsl2def}
A geometric categorical $\sl_2$ action consists of the following data.

\begin{enumerate}
\item A sequence of smooth complex varieties $Y(-N), Y(-N+1), \dots, Y(N-1), Y(N)$ over $\k$ (equipped with $\C^\times$-actions) 
\item Fourier-Mukai kernels 
$$\sE^{(r)}(\l) \in  D(Y(\l-r) \times Y(\l+r)) \text{ and } \sF^{(r)}(\l) \in  D(Y(\l+r) \times Y(\l-r))$$
(which are $\k^\times$ equivariant).
We will usually write $\sE(\l)$ for $\sE^{(1)}(\l)$ and $\sF(\l)$ for $\sF^{(1)}(\l)$ while one should think of $\sE^{(0)}(\l)$ and $\sF^{(0)}(\l)$ as $\O_\Delta$.
\item For each $Y(\l)$ a flat deformation $\tY(\l) \rightarrow \bA^1_\k$ carrying a $\k^\times$-action covering the action $x \mapsto t^2 x$ (where $t \in \k^\times$) on the base. 
\end{enumerate}

On this data we impose the following additional conditions.

\begin{enumerate}
\item Each (graded piece of the) $\Hom$ space between two objects in $D(Y(\l))$ is finite dimensional. In particular, this means that $\End(\O_{Y(\l)}) = \k \cdot I$.
\item 
$\sE^{(r)}(\l)$ and $\sF^{(r)}(\l)$ are left and right adjoints of each other up to shift. More precisely
\begin{enumerate}
\item $\sE^{(r)}(\l)_R = \sF^{(r)}(\l)[r\l]\{-r\l\}$ or equivalently $\sF^{(r)}(\l)_L = \sE^{(r)}(\l)[r\l]\{-r\l\}$
\item $\sE^{(r)}(\l)_L = \sF^{(r)}(\l)[-r\l]\{r\l\}$ or equivalently $\sF^{(r)}(\l)_R = \sE^{(r)}(\l)[-r\l]\{r\l\}$
\end{enumerate}

\item 
At the level of cohomology of complexes we have
$$\H^*(\sE(\l+r) * \sE^{(r)}(\l-1)) \cong \sE^{(r+1)}(\l) \otimes_{\k} H^\star(\p^{r}).$$

\item 
If $ \l \le 0 $ then 
\begin{equation*}
\sF(\l+1) * \sE(\l+1) \cong \sE(\l-1) * \sF(\l-1)  \oplus \sP
\end{equation*}
where $\H^*(\sP) \cong \O_\Delta \otimes_\k H^\star(\p^{-\l-1})$. 

Similarly, if $\l \ge 0$ then 
\begin{equation*}
\sE(\l-1) * \sF(\l-1) \cong \sF(\l+1) * \sE(\l+1) \oplus \sP'
\end{equation*}
where $\H^*(\sP') \cong \O_\Delta \otimes_\k H^\star(\p^{\l-1})$.

\item  We have
$$\H^*(i_{23*} \sE(\l+1) * i_{12*} \sE(\l-1)) \cong \sE^{(2)}(\l)[-1]\{1\} \oplus \sE^{(2)}(\l)[2]\{-3\}$$
where the $i_{12}$ and $i_{23}$ are the closed immersions
\begin{align*}
i_{12}: Y(\l-2) \times Y(\l) \rightarrow Y(\l-2) \times \tY(\l) \\i_{23}: Y(\l) \times Y(\l+2) \rightarrow \tY(\l) \times Y(\l+2).
\end{align*}

\item If $\l \le 0$ and $k \ge 1$ then the image of $\supp(\sE^{(r)}(\l-r))$ under the projection to $Y(\l)$ is not contained in the image of $\supp(\sE^{(r+k)}(\l-r-k))$ also under the projection to $Y(\l)$. Similarly, if $\l \ge 0$ and $k \ge 1$ then the image of $\supp(\sE^{(r)}(\l+r))$ in $Y(\l)$ is not contained in the image of $\supp(\sE^{(r+k)}(\l+r+k))$.

\item All $\sE^{(r)}$s and $\sF^{(r)}$s are sheaves (i.e. complexes supported in degree zero).
\end{enumerate}

\begin{Remark} It is probably more appropriate to call the above structure a $U_q(\sl_2)$ action, since the $\C^\times$ action provides a second grading which allows for a categorification of the quantum group $U_q(\sl_2)$ (rather than just the Lie algebra $\sl_2$). One may very well choose to ignore the $\C^\times$ action and equivariance requirements in the above, in which case one should also ignore all the $\{\cdot\}$ shifts in the definition of a geometric categorical $\sl_2$ action. For many purposes, such as constructing the equivalence described in \cite{ckl2}, having a geometric categorical $\sl_2$ action without the $\C^\times$ action is sufficient. 
\end{Remark}

\begin{Remark} Having the conditions at the level of cohomology may seem strange but, as we will see, it is often easier to work with the cohomology of a complex than with the complex itself.  A great example of this phenomenon is Lemma \ref{lem:fibreproduct} which we use repeatedly. The statement there holds at the level of cohomology and is simply not true at the level of complexes. To overcome this problem we use the deformations (with the properties described above) to lift isomorphisms at the level of cohomology to isomorphisms at the level of complexes.
\end{Remark}

\subsection{The Main Theorem}

Fix $m, N$ as in section \ref{sec:setup}. We work over the ground field $\k = \C$. For $\l = -N, -N+2, \dots, N-2,N$ define $Y(\l) := Y(k,l)$ and $\tY(\l) := \tY(k,l) \rightarrow \bA^1_\C$ where $\l = l-k$ is the weight and $k+l = N$. The rest of the $Y(\l)$s are taken to be empty. 

For the FM kernels we take 
$$\sE^{(r)}(\l) := \sE^{(r)} \left( \frac{N-\l-r}{2}, \frac{N+\l+r}{2} \right): Y(\l-r) \rightarrow Y(\l+r)$$ 
and 
$$\sF^{(r)}(\l) := \sF^{(r)} \left( \frac{N-\l-r}{2}, \frac{N+\l+r}{2} \right): Y(\l+r) \rightarrow Y(\l-r).$$

\begin{Theorem}\label{thm:main}
Under the identifications above, the spaces $Y(\l)$ and $\tY(\l) \rightarrow \bA^1_\C$ together with kernels $\sE^{(r)}(\l)$ and $\sF^{(r)}(\l)$ define a geometric categorical $\C^\times$-equivariant $\sl_2$ action.
\end{Theorem}

\begin{Corollary}\label{cor:main}
Under the identification above, the kernels $\sE^{(r)}(\l)$ and $\sF^{(r)}(\l)$ satisfy the categorified $\sl_2$ relations. In other words, we have
$$\sF^{(r_2)}(k,l) * \sF^{(r_1)}(k-r_1,l+r_1) \cong \sF^{(r_1+r_2)}(k-r_1,l+r_1) \otimes_\C H^\star(\G(r_1,r_1+r_2))$$
and similarly if we replace $\sF$s by $\sE$s. Moreover, if $k \ge l$ then
\begin{equation*}
\sF(k-1,l+1) * \sE(k-1,l+1) \cong \sE(k,l) * \sF(k,l) \oplus \O_\Delta \otimes_\C H^\star(\p^{k-l-1})
\end{equation*}
while if $ k \le l $ then
\begin{equation*}
\sE(k,l) * \sF(k,l) \cong \sF(k-1,l+1) * \sE(k-1,l+1) \oplus \O_\Delta \otimes_\C H^\star(\p^{l-k-1}).
\end{equation*}
\end{Corollary}
\begin{proof}
This follows immediately from Theorem \ref{thm:main} together with the main result of \cite{ckl1}.
The idea is to use Proposition \ref{prop:Ecomps2} in the deformed setup to show that the results in Propositions \ref{prop:Ecomps1} and \ref{prop:comm1} are true on the nose (and not just at the level of cohomology). This can be done whenever one has a geometric categorical $\sl_2$ action -- see \cite{ckl1} for a complete proof.
\end{proof}

\begin{Corollary}\label{cor:kaction}
The functors $ \E^{(r)}, \F^{(r)} $ induce an action of $U_q(\sl_2) $ on $ \oplus_{k+l = N} K(D(Y(k,l)) $.
\end{Corollary}

\section{The equivalence} \label{sec:equiv}
In this section we describe the equivalence of categories
$$
    \T(l,k): D(Y(k,l))\longrightarrow D(Y(l,k))
$$
which arises from the geometric categorical $\sl_2 $ action.  We will show that on the Grothendieck group this equivalence is closely related to the R-matrix isomorphism.

\subsection{Construction of the equivalence}\label{sec:equivalence}
Fix $k \le l$. We form a complex $ \Theta_*(k,l)$ of functors whose terms are
$$ \Theta_s(k,l) = \F^{(l-k+s)}(s)\E^{(s)}(l-k+s)[-s]\{s\} $$
where $ s = 0, \dots, k$. The differential $ \Theta_s \rightarrow \Theta_{s-1} $ is given by the composition of maps
\begin{equation*}
\F^{(l-k+s)}\E^{(s)} \xrightarrow{\iota \iota} \F^{(l-k+s-1)} \F\E \E^{(s-1)} \xrightarrow{\varepsilon} \E^{(l-k+s-1)}\F^{(s-1)}
\end{equation*}
where the morphisms $\iota$, $\varepsilon$ are defined in \cite{ckl1}.

\begin{Theorem}\label{thm:equivalence}
The complex $\Theta_*(k,l)$ has a unique cone $ \T(k,l) $.  This cone gives an equivalence of categories
$$
    \T(k,l): D(Y(k,l))\longrightarrow D(Y(l,k)).
$$ 
The map induced by $ \T(k,l) $ from $ K(D(Y(k,l))) \rightarrow K(D(Y(l,k)))$ coincides with the action of the quantum Weyl group element $ t \in \widehat{U_q(\sl_2)} $.
\end{Theorem}

\begin{proof}
By the main result of \cite{ckl1}, the geometric categorical $\sl_2$ action of this paper induces a 
strong categorical $\sl_2$ action, as defined in \cite{ckl1}.  
Once we have established a strong categorical $\sl_2$ action,  the main result of \cite{ckl2} implies that the cone of the complex $\Theta_*(l,k)$ is unique and that this cone induces an equivalence of categories 
$D(Y(k,l))\simeq D(Y(l,k))$.  The statement about the Grothendieck group is also proved in \cite{ckl2}.
\end{proof}

It follows from dimension considerations that there exists an isomorphism of $ \C[q,q^{-1}] $ modules $ K(D(Y(k,l))) \cong \Lambda_q^k \otimes \Lambda_q^l $, where $ \Lambda_q^k $ is a minuscule representation of $ U_q(\sl_m) $.  We would like to fix such an isomorphism and show that under this isomorphism $ \T(k,l) $ corresponds to the braiding.  For this we will need to examine quantum skew Howe duality.

\subsection{Quantum skew Howe duality}
The following results concerning quantum skew Howe duality have not appeared in the literature, but are known to experts.  In particular, they are modifications of the corresponding results for quantum symmetric Howe duality (also known as the quantum matrix function algebra) which has been studied extensively in \cite{PW}, \cite{TL}, following the work of Manin \cite{M}. 

In writing this section, we were greatly helped by conversations with Valerio Toledano Laredo and Arkady Berenstein.

\subsubsection{The vector space and action}
We begin with the algebra
\begin{equation*}
\Lambda_q(\C^n) = \base \langle X_1, \dots, X_n \rangle / ( X_i^2, \ X_i X_j + q X_j X_i \text{ for } i < j )
\end{equation*}

This is a free $\base $ module with basis $ \{ X_{i_1} \cdots X_{i_k} \}$ for ${i_1 < \cdots < i_k} $ (see for example \cite{PW}).  

$\Lambda_q(\C^n) $ has an action of $ U_q(\sl_n) $ in the following way.  Consider $  V = \spn\{X_1, \dots, X_n\} $, the natural representation of $ U_q(\sl_n) $.  Then $ \Lambda_q(\C^n) = TV/(R) $, where $ R \subset V \otimes V $ is a $ U_q(\sl_n) $ subrepresentation.  Hence $ \Lambda_q(\C^n)$ carries an action of $ U_q(\sl_n) $.  

When we set $ q = 1$, we recover the representation $ \Lambda(\C^n) $ of $ \sl_n $.

When $ n = 2$, we will write $ X,Y $ instead of $ X_1, X_2 $ as the basis for $ \C^2$.

Now we consider the ``exterior quantum matrix algebra'' which is a vector space $ \Lambda_q(\C^m \otimes \C^2) $ over $ \C[q,q^{-1}] $ with basis $ \{Y_{i_1}\cdots Y_{i_k} X_{j_1} \cdots X_{j_l} \}$ for $ i_1 < \cdots < i_k$ and $ j_1 < \cdots< j_l$.  It can also be defined as a quotient of a tensor algebra, but we will not need the algebra structure here.  In fact it is the quadratic dual of the more familiar quantum matrix algebra (see \cite{M}, especially section 8.9).

We have isomorphisms of $ \base $ modules
\begin{equation} \label{eq:isoms}
\Lambda_q(\C^m) \otimes \Lambda_q(\C^m) \leftarrow \Lambda_q(\C^m \otimes \C^2) \rightarrow \Lambda_q(\C^2)^{\otimes m}
\end{equation}
where on the left side of (\ref{eq:isoms})
\begin{equation*}
 Y_{i_1} \cdots Y_{i_k} X_{j_1} \cdots X_{j_l} \mapsto X_{i_1} \cdots X_{i_k} \otimes X_{j_1} \cdots X_{j_l}
\end{equation*}
and on the right side of (\ref{eq:isoms})
\begin{equation*}
 Y_{i_1} \cdots Y_{i_k} X_{j_1} \cdots X_{j_l} \mapsto (-1)^{ \# \{(a,b) : i_a < j_b \} }YX \otimes X \otimes 1 \otimes Y \cdots
\end{equation*}
 where we write $ YX $ in the $ p$th slot if $ p \in \{i_1, \dots, i_k\} $ and $ p \in \{j_1, \dots, j_l\}$, etc.

These isomorphisms allow us to give $ \Lambda_q(\C^m \otimes \C^2) $ the structures of $ U_q(\sl_m) $ and $ U_q(\sl_2) $ modules.  When we set $ q = 1$, we recover $ \Lambda(\C^m \otimes \C^2) $, the exterior algebra of $\C^m \otimes \C^2 $.

\begin{Lemma}
The actions of $ U_q(\sl_m) $ and $ U_q(\sl_2) $ on $ \Lambda_q(\C^m \otimes \C^2)$ commute.
\end{Lemma}

\begin{proof}
By considering the root $ U_q(\sl_2) $ subalgebras of $ U_q(\sl_m)$, we see that it is sufficient to check the result for the case $ m=2$.  That case can be checked by an explicit calculation.
\end{proof}

These $ \base $ vector spaces are graded by setting $deg(X_i)=deg(Y_j)=1$.  We will write $ \Lambda^N_q(\C^m \otimes \C^2) $ for the homogeneous subspace of degree $ N$.  These graded pieces are invariant under the actions of $ U_q(\sl_2) $ and $ U_q(\sl_m) $.  Under the left isomorphism in (\ref{eq:isoms}), we have
\begin{equation*}
\Lambda_q^N(\C^m \otimes \C^2) \rightarrow \oplus_{k+l = N} \Lambda^k_q(\C^m) \otimes \Lambda_q^l(\C^m)
\end{equation*}

\subsubsection{The relation between $ R $ and $ t$}
We now will consider two different operators acting on $ \Lambda_q(\C^m \otimes \C^2)$.  On the one hand, by (\ref{eq:isoms}), $ \Lambda_q(\C^m \otimes \C^2) $ is the tensor square of a representation of $ U_q(\sl_m) $.  Hence we have the braiding
 $$
  \beta_{\Lambda_q(\C^m), \Lambda_q(\C^m)} : \Lambda_q(\C^m \otimes \C^2) \rightarrow \Lambda_q(\C^m \otimes \C^2),
$$  
which is defined as the composition 
$$ 
\beta_{\Lambda_q(\C^m), \Lambda_q(\C^m)} = \flip\circ R, 
$$
where $R$ is the universal $\mathrm{R}$-matrix of $U_q(\sl_m)$ and $\flip$ is the vector space isomorphism which exchanges the factors.
On the other hand by (\ref{eq:isoms}), $ \Lambda_q(\C^m \otimes \C^2) $ is a representation of $ U_q(\sl_2) $ and so we can consider the action of the quantum Weyl group element $ t \in \widehat{U_q(\sl_2)} $ on $ \Lambda_q(\C^m \otimes \C^2)$. (There is some ambiguity in the definition of quantum Weyl group elements.  We use the one whose commutation relations are give by (\ref{eq:comm}) with $m =2 $.)

The following result is analogous to \cite[Theorem 6.5]{TL} which dealt with the analogous question for symmetric products.  Our proof follows Toledano Laredo's method.

\begin{Theorem} \label{th:betaeqt}
For $ v \in \Lambda^k(\C^m) \otimes \Lambda^l(\C^m) $, we have
$ \beta(v) =  (-1)^{kl} q^{k - kl/m} t(v)$
\end{Theorem}

\begin{proof}
Both sides are $ U_q(\sl_m) $-module morphisms.  Hence it suffices to check this statement for $ U_q(\sl_m) $ lowest weight vectors.

The possible highest weights of $ U_q(\sl_m) $ occuring in $ \Lambda_q(\C^m \otimes \C^2) $ are $$ \lambda(i,N) := (0, \cdots, 0, 1, \cdots, 1, 2, \cdots 2) $$ where there are $ i $ 2s and $ N -2i $ 1s.

For each $i \le k,l \le N$ such that $ k+l = N $ there are lowest weight vectors $ v_i^{k,l} \in \Lambda_q^k \otimes \Lambda_q^l $ of lowest weight $ w_0(\lambda(i,N))$. (Here we are writing $ \Lambda_q^k $ for 
$ \Lambda_q^k(\C^m) $.)  These form a basis for the space of lowest weight vectors.

We are free to normalize $ v_i^{k,l} $ and we choose a normalization such that
\begin{equation*}
v_i^{k,l} = Y_{m-k-l+i+1} \cdots Y_{m-l} Y_{m-i+1} \cdots Y_m X_{m-l+1} \cdots X_m + \cdots,
\end{equation*}
where $ \cdots $ is a linear combination of other monomials.

The statement of the theorem now follows from the following two Lemmas which describe the action of $ \beta $ and $ t $ on these vectors $ v_i^{k,l}$.  Note that it is immediate that each of these operators takes $ v_i^{k,l} $ to a multiple of $ v_i^{l,k} $.  What is important is to see that the ratio of these multiples is independent of $ i $.

\begin{Lemma} \label{th:actionofbeta}
$ \beta(v_i^{k,l}) = (-1)^{(l-i)(k-i)} q^{-(l-i)(k-i) + i - kl/m} v_i^{l,k}.$
\end{Lemma}

\begin{Lemma} \label{th:actionoft}
$ t(v_i^{k,l}) = (-1)^{(l-i)(k-i) + kl} q^{-(k-i)(l-i) +i - k} v_i^{l,k}. $
\end{Lemma}

\begin{proof}[Proof of Lemma \ref{th:actionofbeta}]
We will make use of the ``half-twist'' formula for the universal R-matrix, due to Kirillov-Reshtikhin \cite{KR} and independently Levendorskii-Soibelman \cite{LS}.

Let $ w_0 $ denote the long element of $ S_m$ and let $t_{w_0} $ denote the corresponding element of the quantum Weyl group of $ U_q(\sl_m) $.  There is some ambiguity regarding these elements.  Here we are using the inverse to the definition in \cite{peter}.  The element $ t_{w_0} $  satisfies the following commutation relations (see for example \cite[Lemma 5.4]{peter})
\begin{equation} \label{eq:comm}
t_{w_0} F_i = -E_{m-i}K_{m-i}t_{w_0}, \quad t_{w_0} E_i = -K^{-1}_{m-i} F_{m-i} t_{w_0}, \quad t_{w_0} K_i =  K_{m-i}t_{w_0}.
\end{equation}

Let $ v_\lambda $ denote the highest weight vector of an irreducible $ U_q(\sl_m) $ representation of highest weight $ \lambda$ and let $ v_\lambda^{\mathrm{low}} $ denote the corresponding lowest weight vector (which is obtained by applying a maximal chain of divided powers of $ F_i $s to $ v_\lambda$).  Then we have that 
\begin{equation} \label{eq:hightolow}
 t_{w_0} v_\lambda = v_\lambda^{\mathrm{low}} 
 \end{equation} 
 (see \cite[Comment 5.10]{peter}).

The half-twist formula gives us that $$ R = q^{H \otimes H} t_{w_0} \otimes t_{w_0} \Delta(t_{w_0}^{-1})$$
where $ q^{H \otimes H} $ denotes the operator which acts by $ q^{(\mu, \nu)} $ on a tensor $ v\otimes w $ of vectors of weight $ \mu, \nu $.

Now we apply the half-twist formula and write our computation inside $ \Lambda_q^k \otimes \Lambda_q^l$.  Successive application of raising operators $ E_j $ shows that the vector $v_i^{k,l} $ is the lowest weight vector of a subrepresentation whose highest weight vector is
\begin{equation*}
X_1 \cdots X_k \otimes X_1 \cdots X_i X_{l+1} \cdots X_{l+k-i} + \cdots.
\end{equation*}
Thus $ \Delta(t_{w_0}^{-1})(v_i^{k,l}) = X_1 \cdots X_k \otimes X_1 \cdots X_i X_{l+1} \cdots X_{l+k-i} + \cdots
  $, by (\ref{eq:hightolow}).

Similarly, $ X_1 \cdots X_k $ is a highest weight vector in a representation whose lowest weight vector is 
$ X_{m-k+1} \cdots X_m $.  On the other hand, 
\begin{equation*}
X_1 \cdots X_i X_{l+1} \cdots X_{l+k-i} = F_l \cdots F_k X_1 \cdots X_k
\end{equation*}
where there are $ (l-i)(k-i) $ different $F$s in the expression on the right hand side of the above equation.
Thus 
\begin{equation*} 
\begin{aligned}
t_{w_0} ( X_1 \cdots X_i X_{l+1} \cdots X_{l+k-i}) &= t_{w_0}( F_l \cdots F_k (X_1 \cdots X_k)) \\
&= (-1)^{(l-i)(k-i)} E_{m-l} K_{m-l} \cdots E_{m-k} K_{m-k} ( t_{w_0} X_1 \cdots X_k) \\
&= (-1)^{(l-i)(k-i)} q^{-(l-i)(k-i)} X_{m-k-l+i+1} \cdots X_{m-l} X_{m-i+1} \cdots X_m 
\end{aligned}
\end{equation*}
where we have used the commutation relations (\ref{eq:comm}) and also (\ref{eq:hightolow}).

Hence
\begin{align*}
R v_i^{k,l} &= q^{H \otimes H}  t_{w_0} \otimes t_{w_0} \Delta(t_{w_0})(v_i^{k,l})  \\
&= q^{H \otimes H} t_{w_0} \otimes t_{w_0}(X_1 \cdots X_k \otimes X_1 \cdots X_i X_{l+1} \cdots X_{l+k-i} + \cdots) \\
&= q^{H \otimes H}(X_{m-k+1} \cdots X_m \otimes (-1)^{(l-i)(k-i)} q^{-(l-i)(k-i)} X_{m-k-l+i+1} \cdots X_{m-l} X_{m-i+1} \cdots X_m + \cdots) \\
&= (-1)^{(l-i)(k-i)} q^{-(l-i)(k-i) + i - lk/m} X_{m-k+1} \cdots X_m \otimes X_{m-k-l+i+1} \cdots X_{m-l} X_{m-i+1} \cdots X_m + \cdots
\end{align*}
where in the last step we used that $ (w_0 \omega_k , w_0 \lambda(i,N)) = i - lk/m $.

Now applying $\flip$ gives the desired result.
\end{proof}

\begin{proof}[Proof of Lemma \ref{th:actionoft}]
Fix $ i, N $ as above, so that $\lambda(i,N) = (0, \cdots, 0, 1, \cdots, 1, 2, \cdots 2) $, where there are $ i $ 2s and $ N -2i $ 1s.  Since the $ U_q(\sl_2) $ action on $ \Lambda_q(\C^m \otimes \C^2) $ commutes with the 
$ U_q(\sl_m) $ action, the vector space of $ \lambda(i,N) $ lowest weight vectors forms a $ U_q(\sl_2) $ submodule.  So $ \spn\{v_i^{N-i,i}, \dots, v_i^{i,N-i}\} $ is a $ U_q(\sl_2) $ subrepresentation.  Since $ v_i^{k,l} $ is a $U_q(\sl_2)$ weight vector of weight $(l,k)$, we see that this subrepresentation is in fact the irreducible $U_q(\sl_2)$ representation of highest weight $(N-i,i)$.

Let 
$$ \tilde{v}_i^{k,l} = 1 \otimes \cdots \otimes 1 \otimes Y \otimes \cdots \otimes Y \otimes X \otimes \cdots \otimes X \otimes YX \otimes \cdots \otimes YX ,$$
so that under the isomorphism (\ref{eq:isoms}), $ v_i^{k,l} $ is taken to $ (-1)^{i(i-1)/2 + (l-i)i}\tilde{v}_i^{k,l} $.

Now, we claim that 
$$ \tilde{v}_i^{k,l} = F^{(k-i)} \tilde{v}_i^{i, N-i} = E^{(l-i)} \tilde{v}_i^{N-i, i}.$$  
To see this, note that from weight considerations $ F^{(k-i)} \tilde{v}_i^{i,N-i} $ must be a multiple of $ \tilde{v}_i^{k,l} $; so it suffices to show that the coefficient of the ``leading'' monomial in $ F^{(k-i)} \tilde{v}_i^{i, N-i} $ is $ 1$.  Applying the coproduct many times to $ F^{k-i} $, we see that this coefficient can be written as a sum over the symmetric group, where the symmetric group elements which index the summation indicate in which order the $ F$s act on the tensor factors. In particular this leading coeffient equals
\begin{equation*}
1/[k-i]! \sum_{\sigma \in S_{k-i}} \prod_{j=1}^{k-i} q^{2\# \{a < j : \sigma(a) < \sigma(j) \}  - j } = 1.
\end{equation*}
A similar argument holds for $ E^{(l-i)} \tilde{v}_i^{N-i,i} $ (here $[r]!$ denotes the quantum factorial).

It follows that
\begin{equation*}
t \tilde{v}_i^{k,l} = t F^{(k-i)} \tilde{v}_i^{i, N-i} = \frac{1}{[k-i]!} (-E K)^{k-i} \tilde{v}_i^{N-i, i} = (-1)^{l-i} q^{-(k-i)(l-i+1)} \tilde{v}_i^{l,k}
\end{equation*}
where we use that $ N-2i + \cdots + N - 2(i + k-i-1) = (k-i)(N - k -i +1) = (k-i)(l-i+1)$.

Hence $ t(v_i^{k,l}) = (-1)^{i(i-1)/2 + (l-i)i + i(i-1)/2 + (k-i)i + l-i} q^{-(k-i)(l-i+1)} v_i^{l,k}$

Since $$ i(i-1) + (l-i)i + (k-i)i + l-i + kl \equiv (k-i)(l-i) \quad (\mathrm{mod} 2), $$
the result now follows.
\end{proof}

\end{proof}

\subsection{The map on the Grothendieck group}
Now, we are in a position to prove the main theorem of this section which states that on the level of the Grothendieck group, our equivalence recovers the R-matrix.

\begin{Theorem} 
There exists isomorphisms of $ \base $ modules, $ K(D(Y(k,l))) \rightarrow \Lambda_q^k \otimes \Lambda_q^l $ such that the diagram
\begin{equation*}
\begin{CD}
K(D(Y(k,l))) @>>(-1)^{kl}q^{k-kl/m}[\T]> K(D(Y(l,k))) \\
@VVV @VVV \\
\Lambda_q^k \otimes \Lambda_q^l @>>\beta> \Lambda_q^l \otimes \Lambda_q^k
\end{CD}
\end{equation*}
commutes.
\end{Theorem}

\begin{proof}
We abbreviate $ K(k,l) = K(D(Y(k,l)))$ (the Grothendieck group of $D(Y(k,l))$).  

By Corollary \ref{cor:kaction}, $ \oplus_{k + l = N} K(k,l) $ is a $ U_q(\sl_2)$ representation, with each $K(k,l)$ a $U_q(\sl_2)$ weight space.  By the argument of Proposition 7.2 of \cite{ck2}, we have that $\dim K(k,l) = \dim H^*( Y(k,l) )$.  It is also easy to see that
$ \dim H^* (Y(k,l)) = \binom{m}{k} \binom{m}{l} $ using the description of $Y(k,l)$ as a Grassmannian bundle over a Grassmanian.  Thus the $U_q(\sl_2)$ weight space $\dim K(k,l)$  has dimension $\binom{m}{k} \binom{m}{l} $.

Now, $ \oplus_{k + l = N} \Lambda_q^k \otimes \Lambda_q^l $ is also a $ U_q(\sl_2) $ representation with weight spaces $ \Lambda_q^k \otimes \Lambda_q^l$ of dimension $\binom{m}{k} \binom{m}{l} $.

Since $ U_q(\sl_2) $ representations are determined by the dimensions of their weight spaces, we deduce that there exists an isomorphism $ \oplus_{k,l} K(k,l) \rightarrow \oplus_{k,l}  \Lambda_q^k \otimes \Lambda_q^l $ compatible with the $ U_q(\sl_2) $ actions. Fixing such an isomorphism, the theorem now follows from Theorem \ref{th:betaeqt}, since $[\T] = t$ is the quantum Weyl group element.

\end{proof}

\section{Preliminaries}\label{sec:prems}

In this section we prove three technical results we will need later. The first is a lemma (Lemma \ref{lem:fibreproduct}) which somewhat generalizes the standard result that ``cohomology commutes with flat base extension''. We then give a ``train" description of the functors $\E$ and $\F$, that is, we describe $\E$ and $\F$ using a series of pushforward and pullbacks  (Proposition \ref{prop:train}). Finally, in section \ref{sec:basiccalcs} we compute the canonical bundles of certain varieties while using their structure as iterated Grassmannian bundles.

\subsection{Cohomology and base change}

Many of the proofs in this section involve taking fibre products, and as a result we will make frequent use of the following lemma.

\begin{Lemma}\label{lem:fibreproduct} Let $Y$ be a smooth variety and let 
$$f_1: X_1 \rightarrow Y \text{ and } f_2: X_2  \rightarrow Y$$ 
be maps satisfying the following:
\begin{enumerate}
\item $f_1$ and $f_2$ are flat over their images $f_1(X_1),f_2(X_2)$.
\item $f(X_1)$ and $f(X_2)$ are local complete intersections.
\item There is a local complete intersection $Y'\subset Y$ containing $f(X_1)$ and $f(X_2)$ 
such that the intersection
$f(X_1)\cap f(X_2)$ in $Y'$ is of the expected dimension. 
\end{enumerate}
Then the kernel $\sK \in D(X_1 \times X_2)$ inducing the functor $f_2^* \circ {f_1}_*: D(X_1) \rightarrow D(X_2)$ has cohomology given by
$$\H^{-s}(\sK) \cong p^*(\wedge^s N_{Y'/Y}^{\vee})$$
where $p: X_1 \times_Y X_2 \rightarrow Y'$ is the natural projection. 

Moreover, if $X_1,X_2$ and $Y$ carry $\C^\times$ actions and the maps $f_1,f_2$ are $\C^\times$-equivariant then the same is true $\C^\times$-equivariantly. 
\end{Lemma}
\begin{proof}
Consider $\O_{X_1} \in D(X_1 \times Y)$ and $\O_{X_2} \in D(Y \times X_2)$ via the inclusions 
$X_i \hookrightarrow X_i \times Y$.  Then the kernel $\sK$ is by definition the convolution 
$\O_{X_2} * \O_{X_1}$.  Thus
\begin{eqnarray*}
\sK &\cong& \pi_{13*}(\pi_{12}^* \O_{X_1} \otimes \pi_{23}^* \O_{X_2}).
\end{eqnarray*}
Now both $V_1 := \pi_{12}^{-1}(X_1)$ and $V_2 := \pi_{23}^{-1}(X_2)$ have codimension $\dim(Y)$ 
in $X_1 \times Y \times X_2$. On the other hand, if we consider the map $V_1 \cap V_2 \rightarrow Y$ the image has codimension 
$$\codim(f_1(X_1),Y) + \codim(f_2(X_2),Y) - \codim(Y',Y).$$ 
Moreover, since $f_1$ and $f_2$ are flat over their image, the fibers of this map all have dimension 
$\dim(f_1) + \dim(f_2)$.  Thus $V_1 \cap V_2$ has dimension 
$$-\dim(Y') + \dim(f_1(X_1)) + \dim(f_2(X_2)) + \dim(f_1) + \dim(f_2) = \dim(X_1) + \dim(X_2) - \dim(Y')$$
Equivalently, $\codim(V_1\cap V_2,X_1\times Y\times X_2)=\dim(Y) + \dim(Y')$; note that this may be less than the expected codimension, which is $2\dim(Y)$.  Since both $V_1$ and $V_2$ lie inside 
$$Z := X_1 \times Y' \times X_2 \subset X_1 \times Y \times X_2,$$
we find that $V_1\cap V_2$ is of the expected dimension in $Z$. 

Denote by $i$ and $j$ the sequence of inclusions $V_1 \hookrightarrow Z \hookrightarrow X_1 \times Y \times X_2$. Then 
\begin{equation*}
\pi_{12}^* \O_{X_1} \otimes \pi_{23}^* \O_{X_2} \cong (j \circ i)_* (j \circ i)^* \O_{V_2} \cong j_* i_* i^* (j^* j_* \O_{V_2}).
\end{equation*}
Now 
$$j_* j^* j_* \O_{V_2} = \O_{V_2} \otimes \O_Z \cong \bigoplus_s N_{Z/X_1 \times Y \times X_2}^\vee|_{V_2}[s] \cong \bigwedge^s (\pi^* N_{Y'/Y}^\vee)|_{V_2}[s]$$
where $\pi: Z \rightarrow Y'$ is the natural projection. Then 
$$\H^{-s}(j^* j_* \O_{V_2}) \cong (\pi^* \wedge^s N_{Y'/Y}^\vee)|_{V_2}.$$
Finally, $i^*(\pi^* N_{Y'/Y}|_{V_2}) \cong (\pi^* N_{Y'/Y})|_{V_1 \cap V_2}$ since $V_2$ and $V_1$ intersect in the expected dimension in $Z$ and $ N_{Y'/Y}$ is locally free. Thus 
$$\H^{-s}(\pi_{12}^* \O_{X_1} \otimes \pi_{23}^* \O_{X_2}) = (\pi^* \wedge^s N_{Y'/Y}^\vee)|_{\pi_{12}^{-1}(X_1) \cap \pi_{23}^{-1}(X_2)}.$$

The projection $\pi_{13}$ maps $\pi_{12}^{-1}(X_1) \cap \pi_{23}^{-1}(X_2)$ one-to-one onto $X_1 \times_Y X_2 \subset X_1 \times X_2$ and thus 
$$\H^{-s}(\O_{X_2} * \O_{X_1}) \cong p^*(\wedge^s N_{Y'/Y}^\vee).$$
The equivariant case is the same. 
\end{proof}

\subsection{Train descriptions}
For nonnegative integers $a,b,c$, let
\begin{equation*}
Y(a,b,c) := \{ \C[[z]]^m = L_0 \xrightarrow{a} L_1 \xrightarrow{b} L_2 \xrightarrow{c} L_3 \subset \vect : z L_i \subset L_{i-1} \}.
\end{equation*}
For $ i = 1,2 $, let $ X(a,b,c)_i $ denote the subvariety of $ Y(a,b,c) $ where $ z L_{i+1} \subset L_{i-1} $.  Note that the map 
 $ X(a,b,c)_1 \rightarrow Y(a+b, c) $ which forgets $L_1$ describes $X(a,b,c)_1$ as the total space of a 
 $ \G(a, a+b) $ bundle; from this it follows that the inclusion $ X(a,b,c)_1 \hookrightarrow Y(a,b,c) $ has codimension $ ab $.  Similarly, $X(a,b,c)_2$ is a codimension $bc$ subvariety of $ Y(a,b,c) $ and a 
 $ \G(b,b+c) $ bundle over $Y(a,b+c) $.

\begin{Lemma}\label{lem:transverse1}
The scheme-theoretic intersection of $X(a,b,c)_1$ and $X(a,b,c)_2$ inside $Y(a,b,c)$ is the reduced, smooth variety 
$$\{L_0 \xrightarrow{a} L_1 \xrightarrow{b} L_2 \xrightarrow{c} L_3: zL_2 \subset L_0, zL_3 \subset L_1 \}$$
which has the expected dimension.
\end{Lemma}
\begin{proof}
First notice that $X(a,b,c)_1$ is codimension $ab$ inside $Y(a,b,c)$ -- this can be seen by calculating the dimension of $X(a,b,c)_1$ which is a $\G(a,a+b)$ bundle over $Y(a+b,c)$. Similarly, $X(a,b,c)_2$ is codimension $bc$ inside $Y(a,b,c)$. Thus the expected codimension of their intersection is $ab+bc$. 

Now, forgetting $L_3$ first, then $L_1$, and finally $L_2$ shows that $X(a,b,c)_1 \cap X(a,b,c)_2$ has codimension
$$\dim Y(a,b,c) - c(m-b-c) - ab - (a+b)(m-a-b) = ab+bc.$$
Thus $X(a,b,c)_1 \cap X(a,b,c)_2$ has the expected dimension. 

What remains to check is that the scheme-theoretic intersection is reduced.  To see this, it is enough to show that $ X(a,b,c)_1 $ and $ X(a,b,c)_2 $ meet transversely.    

By Theorem 3.1 of \cite{ck2}, we have a diffeomorphism $$ Y(a,b,c) \rightarrow Y' := \G(a,m) \times \G(b,m) \times \G(c,m). $$  Under this diffeomorphism $ X(a,b,c)_1 $ is carried to $$ X'_1 := \{(W_1, W_2, W_3) \in Y' : W_1 \perp W_2 \}$$ and $X(a,b,c)_2 $ is carried to $$X'_2:= \{(W_1, W_2, W_3) \in Y' : W_2 \perp W_3 \}.$$  So it suffices to check that $X'_1 $ and $X'_2$ intersect transversely which can be done by an explicit local calculation.
\end{proof}

Consider the following diagram of maps 
\begin{equation} \label{diag:train}
\xymatrix{
X(k,r,l-r)_1 \ar[r]^{i_1} \ar[d]_{q_1} &Y(k,r,l-r)  &X(k,r,l-r)_2 \ar[l]_{i_2} \ar[d]_{q_2} \\
Y(k,l) & & Y(k+r,l-r)
}
\end{equation}

As a computational tool we will use the following ``train'' descriptions for the functors $ \E, \F $.

\begin{Proposition}\label{prop:train}
We have $ \F^{(r)}(k,l)(\cdot) = {q_2}_* i_2^* \big( {i_1}_* q_1^* (\cdot) \otimes \det(L_2/L_1)^{l-k-r} \{r(l-r)\} \big) $. 
\end{Proposition}

As a shorthand for the statement of this proposition, we will say that the functor $ \F^{(r)}(k,l) : D(Y(k,l)) \rightarrow D(Y(l,k)) $ is given by the sequence of varieties and maps in the diagram (\ref{diag:train}) along with the line bundle $\det(L_2/L_1)^{l-k-r} \{r(l-r)\}$ on $ Y(k,r,l-r)$.

\begin{Remark} It turns out that even more is true. Each injection/surjection $X \rightarrow Y$ in the above diagram can be viewed as a FM transform with kernel $\O_X \in X \times Y$. Convolving these four kernels we obtain a kernel $\sK \in D(Y(k,l) \times Y(k+r,l-r))$.  In fact $\sK \cong \sF^{(r)}(k,l)$. 
\end{Remark}
\begin{proof}
By Lemma \ref{lem:transverse1} the intersection $X(k,r,l-r)_1\cap X(k,r,l-r)_2$ is smooth of the expected dimension, with 
$$X(k,r,l-r)_1 \times_{Y(k,r,l-r)} X(k,r,l-r)_2 \cong \{L_\bullet: zL_2 \subset L_2, zL_3 \subset L_1 \} = W^r(k,l).$$
Thus, by Lemma \ref{lem:fibreproduct} the kernel inducing the functor 
$$(\cdot) \mapsto i_2^*(i_{1*}(\cdot) \otimes \det(L_2/L_1)^{l-k-r}): D(X(k,r,l-r)_1) \rightarrow D(X(k,r,l-r)_2)$$
is $\det(L_2/L_1)^{l-k-r}|_{W^r(k,l)} \cong \sF^{(r)}(k,l)$, thought of as a kernel in 
$D(X(k,r,l-r)_1 \times X(k,r,l-r)_2)$. The result now follows. 

\end{proof}

\subsection{Basic Calculations}\label{sec:basiccalcs}
Recall that the map which forgets $L_2$ describes $Y(k,l)$ as a $\G(l,m)$ bundle over $\G(k,m)$.
There is an analogous iterated Grassmannian bundle description of $W^r(k,l)$ given by first forgetting $ L_2$, then $ L_1$, then finally $L'_1$. It follows from this that $$\dim W^r(k,l) = 1/2(\dim Y(k,l) + \dim Y(k+r, l-r)).$$ 
We will use these descriptions in the lemmas below to compute the canonical bundle of $Y(k,l)$ and $W^r(k,l)$.

\begin{Lemma}\label{lem;projcalc}
Let $\pi_1: W(k,l)\rightarrow Y(k,l)$ and $\pi_2:W(k,l)\rightarrow Y(k+1,l-1)$ be the two projections.
\begin{enumerate}
\item  For $k < l$, $\pi_1$ is generically a $ \p^{l-k-1}$ bundle.
\item The image of $ \pi_2(W(k,l)) \subset Y(k+1, l-1) $ is a subvariety of codimension $ l-k-1 $.
\end{enumerate} 
\end{Lemma}
\begin{proof}
This follows immediately from the definition of $W(k,l)$.
\end{proof}

\begin{Lemma}\label{lem:zcalc}
For $Y=Y(k,l)$ or $Y=W^r(k,l)$, we have
\begin{enumerate}
\item $\det(z^{-1}L_i/L_i) \cong \O_{Y} \{2b_i\}$
\item $\det(z^{-1}L_i/L_{i+1}) \cong \det(L_{i+1}/L_i)^\vee \{2b_i+2m\}$ 
\end{enumerate}
where $b_i = \mbox{rank}(L_i/L_0)$. 
\end{Lemma}
\begin{proof}
(i). From the exact sequence $0 \rightarrow z^{-1}L_0/L_0 \rightarrow z^{-1}L_i/L_0 \xrightarrow{z} L_i/L_0 \{2\} \rightarrow 0$ we get
\begin{align*}
\det(z^{-1}L_i/L_i) 
&\cong \det(z^{-1}L_i/L_0) \otimes \det(L_i/L_0)^\vee \\
&\cong \det(z^{-1}L_0/L_0) \otimes \O_{Y_\beta} \{2 b_i\} \\
&\cong \O_{Y_\beta} \{2b_i + 2m\}
\end{align*}
where, to get the last equality, we use the isomorphism $z^{-1}L_0/L_0 \cong \O_{Y_\beta}^{\oplus m} \{2\}$.

(ii). Similarly, using the exact sequence $0 \rightarrow L_{i+1}/L_i \rightarrow z^{-1}L_i/L_i \rightarrow z^{-1}L_i/L_{i+1} \rightarrow 0$ we have
\begin{equation*}
\det(z^{-1}L_i/L_{i+1}) \cong \det(z^{-1}L_i/L_i) \otimes \det(L_{i+1}/L_i)^\vee
\cong \det(L_{i+1}/L_i)^\vee \{2b_i+2m\}.
\end{equation*}
\end{proof}

\begin{Lemma}\label{lem:cancalc} We have the following canonical bundle isomorphisms
\begin{enumerate}
\item $\omega_{Y(k,l)} \cong \det(L_2/L_0)^m \{-2m(k+l)-2kl\}$
\item $\omega_{W^r(k,l)} \cong \det(L_2/L_0)^{m} \det(L_2/L_1')^{-r} \det(L_1/L_0)^r \det(L_1'/L_1)^{l-k-r} \{-2m(k+l)-2k(l-r)\}.$
\end{enumerate}
\end{Lemma}
\begin{proof}
The basic fact we will use repeatedly is that the tangent bundle of the Grassmanian $\G(k,n)$ is $\Hom(S,Q)$ where $S$ is the tautological bundle and $Q$ is the corresponding quotient bundle. This means that 
$$\omega_{\G(k,n)} \cong \det(S \otimes Q^\vee) \cong \det(S)^{\dim(Q)} \otimes \det(Q)^{-\dim(S)}.$$

(i). Using the usual $\G(l,m)$ fibration $p: Y(k,l) \rightarrow \G(k,m)$ given by forgetting $L_2$ we have
\begin{eqnarray*}
\omega_{Y(k,l)}
&\cong& \omega_p \otimes \omega_{\G(k,m)} \\
&\cong& \det((L_2/L_1) \otimes (z^{-1}L_1/L_2)^\vee) \otimes \det((L_1/L_0) \otimes (z^{-1}L_0/L_1)^\vee) \\
&\cong& \det(L_2/L_1)^{m-l} \det(L_2/L_1)^{l} \{-l(2k+2m)\} \det(L_1/L_0)^{m-k} \det(L_1/L_0)^k \{-k(2m)\} \\
&\cong& \det(L_2/L_1)^{m} \det(L_1/L_0)^{m} \{-2m(k+l)-2kl\} \\
&\cong& \det(L_2/L_0)^m \{-2m(k+l)-2kl\}
\end{eqnarray*}
where we used Lemma \ref{lem:zcalc} to obtain the third isomorphism. 

(ii). We first use the $\G(l-r,m-r)$ fibration 
$$W^r(k,l) \rightarrow \{L_0 \xrightarrow{k} L_1 \xrightarrow{r} L_1': zL_1' \subset L_0 \}$$ 
given by forgetting $L_2$. We then forget $L_1$ which gives a $\G(k,r)$ fibration over $\G(k+r,m)$. Using these iterated fibrations (like in the calculation of $\omega_{Y(k,l)}$ above) gives us
\begin{eqnarray*}
\omega_{W^r(k,l)} 
&\cong& \det((L_2/L_1') \otimes (z^{-1}L_1/L_2)^\vee) \otimes \det((L_1/L_0) \otimes (L_1'/L_1)^\vee) \otimes \det((L_1'/L_0) \otimes (z^{-1}L_0/L_1')^\vee) \\
&\cong& \left( \det(L_2/L_1')^{m-l} \det(L_2/L_1)^{l-r} \{-(2k+2m)(l-r)\} \right) \otimes \left( \det(L_1/L_0)^r \det(L_1'/L_1)^{-k} \right) \otimes \\
& & \otimes \left( \det(L_1'/L_0)^{m-k-r} \det(L_1'/L_0)^{k+r} \{-2m(k+r)\} \right) \\
&\cong& \det(L_2/L_0)^{m} \det(L_2/L_1')^{-r} \det(L_1/L_0)^r \det(L_1'/L_1)^{l-k-r} \{-2m(k+l)-2k(l-r)\}
\end{eqnarray*} 
where we used Lemma \ref{lem:zcalc} to obtain the second isomorphism. 
\end{proof}

\section{Properties of the functors}\label{sec:proofs}

In this section we study properties of the functors $\E(k,l)$ and $\F(k,l)$ induced by kernels $\sE(k,l)$ and $\sF(k,l)$. Putting together the results of Propositions \ref{prop:Eadj}, \ref{prop:Ecomps1}, \ref{prop:Ecomps2} and \ref{prop:comm1} proves the main Theorem \ref{thm:main}. 

\subsection{Adjunctions}\label{sec:adjunctions}
The following adjunction relations hold. 

\begin{Proposition} \label{prop:Eadj} We have
\begin{enumerate}
\item $\sF^{(r)}(k,l)_R = \sE^{(r)}(k,l)[r(k-l+r)]\{-r(k-l+r)\}$
\item $\sF^{(r)}(k,l)_L = \sE^{(r)}(k,l)[r(l-k-r)]\{-r(l-k-r)\}$
\end{enumerate}
\end{Proposition}
\begin{proof}
We have
\begin{align*}
&\sF^{(r)}(k,l)_R = \sF^{(r)}(k,l)^\vee \otimes \pi_2^* \omega_{Y(k,l)} [\dim Y(k,l)] \\
&\cong \omega_{W^r(k,l)} \otimes \det(L_1'/L_1)^{k+r-l} \{-r(l-r)\} \otimes \pi_1^* \omega_{Y(k+r,l-r)}^\vee [\dim W^r(k,l) - \dim Y(k+r,l-r)] \\
&\cong \left( \det(L_2/L_0)^m \det(L_2/L_1')^{-r} \det(L_1/L_0)^r \det(L_1'/L_1)^{l-k-r} \{-2m(k+l)-2k(l-r)\} \right)|_{W^r(k,l)} \\ 
& \otimes \det(L_1'/L_1)^{k+r-l} \otimes (\det L_2/L_0)^{-m} \{2m(k+l)+2(k+r)(l-r)\} [r(k-l+r)] \{-r(l-r)\} \\
&\cong \sE^{(r)}(k,l)[r(k-l+r)]\{-r(k-l+r)\}.
\end{align*}
where we use that $ \dim W^r(k,l) - \dim Y(k+r,l-r) = \frac{1}{2}( \dim Y(k,l) - \dim Y(k+r, l-r) ) = r(k-l+r) $ and the expression for $\omega_{Y(k+r,l-r)}$ and $\omega_{W^r(k,l)}$ from Lemma \ref{lem:cancalc}. 
The second equality follows similarly.
\end{proof}

\subsection{Composition of $\E$s or $\F$s}

In this section we study the composition of $\F$s (or equivalently of $\E$s). We begin with the ``non-deformed'' version. 

\begin{Proposition}\label{prop:Ecomps1} 
We have
$$\H^*(\sF^{(r_2)}(k,l) * \sF^{(r_1)}(k-r_1,l+r_1)) \cong \sF^{(r_1+r_2)}(k-r_1,l+r_1) \otimes H^\star(\G(r_1,r_1+r_2))$$
and similarly if we replace the $\sF$s by $\sE$s. 
\end{Proposition}
\begin{proof}

We will work at the level of functors keeping in mind that the corresponding statements hold at the level of kernels. By Proposition \ref{prop:train}, the composition $\F^{(r_2)}(k,l) \circ \F^{(r_1)}(k-r_1,l+r_1)$ is given by the following sequence of maps
\begin{equation}\label{eq:diagram1}
\xymatrix{
X(k-r_1,r_1,l)_2 \ar[r]^{i_1} \ar[d]_{q_1} &Y(k-r_1,r_1,l)  &X(k-r_1,r_1,l)_1 \ar[l]_{i_2} \ar[dr]_{q_2} \\ 
Y(k-r_1,l+r_1) & & & Y(k,l) 
}
\end{equation}
\begin{equation*} 
\xymatrix{
& X(k,r_2,l-r_2)_2 \ar[r]^{i_1'} \ar[dl]_{q_1'} &Y(k,r_2,l-r_2) &X(k,r_2,l-r_2)_1 \ar[l]_{i_2'} \ar[d]_{q_2'} \\
Y(k,l) & & & Y(k+r_2, l-r_2).
}
\end{equation*} 
Here $Y(k-r_1,r_1,l)$ and $Y(k,r_2,l-r_2)$ carry the line bundles 
$$\det(L_2/L_1)^{l-k+r_1}\{ r_1l \} \text{ and } \det(L_2/L_1)^{l-k-r_2} \{r_2(l-r_2)\}.$$ 
As in the case of Proposition \ref{prop:train}, more is true- the convolution 
$\sF^{(r_2)}(k,l) * \sF^{(r_1)}(k-r_1,l+r_1)$ is actually equal to the convolution of the kernels given by the sequence of maps above- though we prefer notationally to work with the maps instead of with the kernels.          

Using the maps $q_2$ and $q_1'$ above, we construct the fibre product 
\begin{eqnarray*}
B &:=& X(k-r_1,r_1,l)_1 \times_{Y(k,l)} X(k,r_2,l-r_2)_2 \\
&\cong& \{L_0 \xrightarrow{k-r_1} L_1 \xrightarrow{r_1} L_2 \xrightarrow{r_2} L_3 \xrightarrow{l-r_2} L_4: zL_2 \subset L_0 \text{ and } zL_4 \subset L_2\}
\end{eqnarray*}
We also have fibre products
\begin{eqnarray*}
A_1 &:=& X(k-r_1,r_1,l)_2 \times_{Y(k-r_1,r_1,l)} B \\
&\cong& \{L_0 \xrightarrow{k-r_1} L_1 \xrightarrow{r_1} L_2 \xrightarrow{r_2} L_3 \xrightarrow{l-r_2} L_4: zL_2 \subset L_0 \text{ and } zL_4 \subset L_1\}
\end{eqnarray*}
and
\begin{eqnarray*}
A_2 &:=& B \times_{Y(k,r_2,l-r_2)} X(k,r_2,l-r_2)_1 \\
&\cong& \{L_0 \xrightarrow{k-r_1} L_1 \xrightarrow{r_1} L_2 \xrightarrow{r_2} L_3 \xrightarrow{l-r_2} L_4: zL_3 \subset L_0 \text{ and } zL_4 \subset L_2\}.
\end{eqnarray*}
By Lemma \ref{lem:transverse1} $X(k-r_1,r_1,l)_2$ and $X(k-r_1,r_1,l)_1$ are smooth and intersect in the right dimension inside $Y(k-r_1,r_1,l)$ (similarly for $X(k,r_2,l-r_2)_2$ and $X(k,r_2,l-r_2)_1$ inside $Y(k,r_2,l-r_2)$.) Since $q_2$ and $q_1'$ are flat we can use Lemma \ref{lem:fibreproduct} to replace diagram \ref{eq:diagram1} above by the diagram
\begin{equation*}
\xymatrix{
A_1 \ar[r]^{i''_1} \ar[d]_{\tilde{q}_1} &B &A_2 \ar[l]_{i''_2} \ar[d]_{\tilde{q}_2} \\
X(k-r_1,r_1,l)_2 \ar[d]_{q_1} & & X(k,r_2,l-r_2)_1 \ar[d]_{q_2'} \\
Y(k-r_1,l+r_1) & & Y(k+r_2,l-r_2)
}
\end{equation*}
Here $B$ is equipped with the product of the pullback of the line bundles from $Y(k-r_1,r_1,l)$ and $Y(k,r_2,l-r_2)$, which is the line bundle $\det(L_2/L_1)^{l-k+r_1} \otimes \det(L_3/L_2)^{l-k-r_2} \{r_1l+r_2(l-r_2)\}$.  

A straightforward calculation shows that $\dim(B) = \dim Y(k,l) + r_1(k-r_1) + r_2(l-r_2)$. Calculating the dimension of $A_1$ we find that $i''_1$ is a codimension $r_1l$ embedding while $i''_2$ is codimension $r_2k$. On the other hand, the intersection
$$A_1 \cap A_2 = \{L_0 \xrightarrow{k-r_1} L_1 \xrightarrow{r_1} L_2 \xrightarrow{r_2} L_3 \xrightarrow{l-r_2} L_4: zL_3 \subset L_0 \text{ and } zL_4 \subset L_1\}$$
has codimension $r_1l+r_2k+r_1r_2$ inside $B$. Thus $A_1$ and $A_2$ meet in the expected dimension inside
$$D = \{L_\cdot: zL_2 \subset L_0 \text{ and } zL_4 \subset L_2 \text{ and } zL_3 \subset L_2\} \subset B,$$
which has codimension $r_1r_2$. $D$ is carved out inside $B$ by the section $z:L_3/L_2 \rightarrow L_2/L_1 \{2\}$. 

It follows from Lemma \ref{lem:fibreproduct} that the kernel for ${i''_2}^* \circ {i''_1}_*$ has cohomology in degree $-s$ given by 
$$\wedge^s (L_3/L_2 \otimes (L_2/L_1)^\vee \{-2\})|_{A_1 \cap A_2}.$$ 
In other words, we have simplified $\sF^{(r_2)}(k,l) * \sF^{(r_1)}(k-r_1,l+r_1)$ to a kernel $\sK \in D(A_1 \times A_2)$ whose cohomology is
$$\H^{-s}(\sK) \cong \wedge^s (L_3/L_2 \otimes (L_2/L_1)^\vee \{-2\}) \otimes \det(L_2/L_1)^{l-k+r_1} \otimes \det(L_3/L_2)^{l-k-r_2} \{r_1l+r_2(l-r_2)\} |_{A_1 \cap A_2}.$$
(The corresponding functor is $(\cdot) \mapsto p_{2*}(p_1^*(\cdot) \otimes \sK)$ where $p_1$ and $p_2$ are the two projections from $A_1 \times A_2$ to $Y(k-r_1,l+r_1)$ and $Y(k+r_2,l-r_2)$ respectively.)

Next we should calculate 
$${\pi_2}_*(\pi_1^*(\cdot) \otimes \wedge^s  (L_3/L_2 \otimes (L_2/L_1)^\vee \{-2\}) \otimes \det(L_2/L_1)^{l-k+r_1} \otimes \det(L_3/L_2)^{l-k-r_2} \{r_1l+r_2(l-r_2)\} )$$ 
where $\pi_1$ and $\pi_2$ are the two natural morphisms from $A_1 \cap A_2$ to $Y(k-r_1,l+r_1)$ and $Y(k+r_2,l-r_2)$ respectively.

To do this notice that both $\pi_1$ and $\pi_2$ factor through the map 
\begin{equation}\label{eq:grassfibre}
\pi: A_1 \cap A_2 \rightarrow \{L_0 \xrightarrow{k-r_1} L_1 \xrightarrow{r_1+r_2} L_3 \xrightarrow{l-r_2} L_4: zL_3 \subset L_0, \ zL_4 \subset L_1\} = W^{r_1+r_2}(k-r_1,l+r_1)
\end{equation}
which forgets $L_2$. The map $\pi$ is a $\G(r_1,r_1+r_2)$ Grassmannian bundle with relative cotangent bundle $\Omega^1_{\pi} = (L_2/L_1) \otimes (L_3/L_2)^\vee$.  Thus
\begin{eqnarray*}
\wedge^s (L_3/L_2 &\otimes& (L_2/L_1)^\vee \{-2\}) \otimes \det(L_2/L_1)^{l-k+r_1} \otimes \det(L_3/L_2)^{l-k-r_2}) \{r_1l+r_2(l-r_2)\} \\
&\cong& \wedge^s  (L_3/L_2 \otimes (L_2/L_1)^\vee) \{-2s\} \otimes \omega_{\pi} \otimes \det(L_3/L_1)^{l-k+r_1-r_2} \{r_1l+r_2(l-r_2)\} \\
&\cong& \wedge^{r_1r_2-s} \Omega^1_{\pi} \otimes \det(L_3/L_1)^{l-k+r_1-r_2} \{-2s+r_1l+r_2(l-r_2)\}.
\end{eqnarray*}
Therefore
\begin{eqnarray*}
&& \pi_* \left( \wedge^s (L_3/L_2 \otimes (L_2/L_1)^\vee \{-2\})[s] \otimes \det(L_2/L_1)^{l-k+r_1} \otimes \det(L_3/L_2)^{l-k-r_2} \{r_1l+r_2(l-r_2)\} \right) \\ 
&\cong& \pi_*(\Omega^{r_1r_2-s}_{\pi}[s] \otimes \det(L_3/L_1)^{l-k+r_1-r_2}) \{-2s+r_1l+r_2(l-r_2)\} \\
&\cong& \O_{W^{r_1+r_2}(k-r_1,l+r_1)} \det(L_3/L_1)^{l-k+r_1-r_2} \{(r_1+r_2)(l-r_2)\} \\
& & \otimes_\C H^{2r_1r_2-2s}(\G(r_1,r_1+r_2))[-r_1r_2+2s]\{r_1r_2-2s\} \\
&\cong& \sF^{(r_1+r_2)}(k-r_1,l+r_1) \otimes_\C H^{r_1r_2-s}(\G(r_1,r_1+r_2))[-r_1r_2+2s]\{r_1r_2-2s\}
\end{eqnarray*}
where to obtain the second isomorphism we used that $\oplus_s H^s(\Omega^{s})$ contains all the cohomology of a Grassmannian.

It follows from the above computation that the spectral sequence which computes $\H^*(\pi_* \sK)$ from the cohomology of $\H^*(\sK)$ degenerates at the $E_2$ term and hence 
\begin{eqnarray*}
\H^*(\pi_* \sK) &\cong& \bigoplus_s \sF^{(r_1+r_2)}(k-r_1,l+r_1) \otimes_\C H^{2r_1r_2-2s}(\G(r_1,r_1+r_2))[-r_1r_2+2s]\{r_1r_2-2s\} \\
&\cong& \sF^{(r_1+r_2)}(k-r_1,l+r_1) \otimes_\C H^\star(\G(r_1,r_1+r_2)).
\end{eqnarray*}
This proves the statement for the $\sF$s. 
The analogous result for $\sE$s instead of $\sF$s follows by adjunction. 
\end{proof}

\begin{Remark}
Notice that we only needed to prove Proposition \ref{prop:Ecomps1} in the case $r_1=r_2=1$ (the general case then follows from the formal argument in \cite{ckl1}). However, we proved the general statement since it was no more difficult and somewhat instructive. In particular, it is interesting to see where the cohomology of the Grassmannian shows up geometrically -- namely, via the fibration described in equation (\ref{eq:grassfibre}). 
\end{Remark}

Next we must perform a calculation analogous to that of Proposition \ref{prop:Ecomps1} for the deformed varieties $\tY(k,l)$. The argument is quite similar to that in the proof of \ref{prop:Ecomps1}, but we include it for completeness. Again, we only need to check the case $r=1$ but we cover the more general situation since it is instructive and no more difficult. 

\begin{Proposition}\label{prop:Ecomps2} 
Consider the inclusions 
$$j_{12}: Y(k-r,l+r) \times Y(k,l) \rightarrow Y(k-r,l+r) \times \tY(k,l)$$
$$j_{23}: Y(k,l) \times Y(k+1,l-1) \rightarrow \tY(k,l) \times Y(k+1,l-1).$$
 Then we have
$$\H^*(j_{23*} \sF(k,l) * j_{12*} \sF^{(r)}(k-r,l+r)) \cong \sF^{(r+1)}(k-r,l+r) \otimes_\C (\C[-r]\{r\} \oplus \C[r+1]\{-r-2\})$$
and similarly if we replace $\sF$s by $\sE$s. 
\end{Proposition}
\begin{proof}
We imitate the proof of Proposition \ref{prop:Ecomps1} in the case $(r_1,r_2) = (r,1)$.
Using Proposition \ref{prop:train} the functor corresponding to $j_{23*} \sF(k,l) * j_{12*} \sF(k-r,l+r)$ is the composition of functors induced by the following sequence of maps 
\begin{equation*}
\xymatrix{
X(k-r,r,l)_2 \ar[r]^{i_1} \ar[d]_{q_1} &Y(k-r,r,l)  &X(k-r,r,l)_1 \ar[l]_{i_2} \ar[d]_{q_2} & \\ 
Y(k-r,l+r) & & Y(k,l) \ar[r]^{j_2} & \tY(k,l)
}
\end{equation*}
\begin{equation*} 
\xymatrix{
&X(k,1,l-1)_2 \ar[r]^{i_1'} \ar[d]_{q_1'} &Y(k,1,l-1) & X(k,1,l-1)_1 \ar[l]_{i_2'} \ar[d]_{q_2'} \\
\tY(k,l) & Y(k,l) \ar[l]_{j_2} & & Y(k+1, l-1).
}
\end{equation*} 
In the above diagram, $Y(k-r,r,l)$ and $Y(k,1,l-1)$ carry the line bundles 
$\det(L_2/L_1)^{l-k+r}$ and $\det(L_2/L_1)^{l-k-1}$. 

We construct the fibre product 
\begin{eqnarray*}
B &:=& X(k-r,r,l)_1 \times_{\tY(k,l)} X(k,1,l-1)_2 \\
&\cong& \{L_0 \xrightarrow{k-r} L_1 \xrightarrow{r} L_2 \xrightarrow{1} L_3 \xrightarrow{l-1} L_4: zL_2 \subset L_0 \text{ and } zL_4 \subset L_2\}
\end{eqnarray*}
with respect to $j_2 \circ q_2$ and $j_2 \circ q_1'$. Unfortunately, unlike in the non-deformed setting, $q_1'^* \circ j_2^* \circ j_{2*} \circ q_{2*} \ne p_{2*} \circ p_1^*$ (where $p_i$ are the projections from $B$ to $X(k-r,r,l)_1$ and $X(k,1,l-1)_2$) because neither $j_2 \circ q_2$ nor $j_2 \circ q_1'$ is flat. 

Instead, we claim that
$$q_1'^* \circ j_2^* \circ j_{2*} \circ q_{2*} = p_{2*} \circ j'^* \circ j'_* \circ p_1^*$$
where $j'$ is the codimension one inclusion of $B$ into 
\begin{eqnarray*}
\tB := \{L_0 \xrightarrow{k-r} L_1 \xrightarrow{r} L_2 \xrightarrow{1} L_3 \xrightarrow{l-1} L_4; x \in \C: \\
(z-x)L_2 \subset L_0 \text{ and } zL_4 \subset L_2\}.
\end{eqnarray*}
To see this we use the fibre diagram
\begin{equation*}
\xymatrix{
X(k-r,r,l)_1 \ar[r]^{\tilde{j}_2} \ar[d]_{q_2} & \tilde{X}(k-r,r,l)_1 \ar[d]_{\tilde{q}_2} \\
Y(k,l) \ar[r]^{j_2} & \tY(k,l)
}
\end{equation*}
where 
\begin{eqnarray*}
\tilde{X}(k-r,r,l)_1 := \{L_0 \xrightarrow{k-r} L_1 \xrightarrow{r} L_2 \xrightarrow{l} L_3; x \in \C; \\
(z-x)L_2 \subset L_0 \text{ and } zL_3 \subset L_2 \}.
\end{eqnarray*}
Then 
$$j_2^* \circ j_{2*} \circ q_{2*} = j_2^* \circ \tilde{q}_{2*} \circ \tilde{j}_{2*} = q_{2*} \circ \tilde{j}_2^* \circ \tilde{j}_{2*}$$
where the second equality is by the base change Lemma \ref{lem:fibreproduct}. Now we repeat again with the fibre diagram
\begin{equation*}
\xymatrix{
B \ar[r]^{j'} \ar[d]_{p_1} & \tB \ar[d]_{\tilde{p}_1} \\
X(k-r,r,l)_1 \ar[r]^{\tilde{j}_2} & \tilde{X}(k-r,r,l)_1
}
\end{equation*}
to get that $q_{2*} \circ \tilde{j}_2^* \circ \tilde{j}_{2*} = j'^* \circ j'_* \circ p_1^*$, which proves the above claim. 

Next we consider the fibre products
\begin{eqnarray*}
A_1 &:=& X(k-r,r,l)_2 \times_{Y(k-r,r,l)} B 
\end{eqnarray*}
and
\begin{eqnarray*}
A_2 &:=& B \times_{Y(k,1,l-1)} X(k,1,l-1)_1
\end{eqnarray*}
The images of $B$ and $X(k-r,r,l)_2$ inside $Y(k-r,r,l)$ meet in the expected dimension, as do the images of 
$B$ and $X(k,1,l-1)_1$ inside $Y(k,1,l-1)$. So we may use Lemma \ref{lem:fibreproduct} to reduce to the kernel corresponding to the composition of maps in the following diagram
\begin{equation*}
\xymatrix{
A_1 \ar[r]^{i''_1} \ar[d]_{\tilde{q}_1} & \tB &A_2 \ar[l]_{i''_2} \ar[d]_{\tilde{q}_2} \\
X(k-r,r,l)_2 \ar[d]_{q_1} & & X(k,1,l-1)_1 \ar[d]_{q_2'} \\
Y(k-r,l+r) & & Y(k+1,l-1).
}
\end{equation*}
Here $\tB$ is equipped with the line bundles $\det(L_2/L_1)^{l-k+r} \otimes \det(L_3/L_2)^{l-k-1} \{rl+l-1\})$.

As in the proof of Proposition \ref{prop:Ecomps1}, $A_1$ and $A_2$ meet in the right dimension inside  
$$D = \{L_\cdot: zL_2 \subset L_0 \text{ and } zL_4 \subset L_2 \text{ and } zL_3 \subset L_2\} \subset \tB$$
which has codimension $r+1$. Now, however, $D \subset \tB$ (instead of $D \subset B$) is cut out by the section $z:L_3/L_2 \rightarrow L_3/L_1 \{2\}$ (instead of $z: L_3/L_2 \rightarrow L_2/L_1 \{2\}$). 

By Lemma \ref{lem:fibreproduct} this means that the kernel for $i_2''^* \circ {i_1''}_*$ has cohomology in degree $-s$ given by 
$$\wedge^s (L_3/L_2 \otimes (L_3/L_1)^\vee \{-2\})|_{A_1 \cap A_2}.$$
Thus we have simplified $j_{23*} \sF(k,l) * j_{12*} \sF^{(r)}(k-r,l+r)$ to a kernel $\sK \in D(A_1 \times A_2)$ whose cohomology is 
$$\H^{-s}(\sK) \cong \left( \wedge^s(L_3/L_2 \otimes (L_3/L_1)^\vee \{-2\}) \otimes \det(L_2/L_1)^{l-k+r} \otimes \det(L_3/L_2)^{l-k-1} \{rl+l-1\}\right)|_{A_1 \cap A_2}.$$ 
(The corresponding functor is $(\cdot) \mapsto p_{2*}(p_1^*(\cdot) \otimes \sK)$ where $p_1$ and $p_2$ are the two projections from $A_1 \times A_2$ to $Y(k-r,l+r)$ and $Y(k+1,l-1)$ respectively.)

Now, just as in the proof of Proposition \ref{prop:Ecomps1}, we have the map
$$\pi: A_1 \cap A_2 \rightarrow \{L_0 \xrightarrow{k-r} L_1 \xrightarrow{r+1} L_3 \xrightarrow{l-1} L_4: zL_3 \subset L_0 \text{ and } zL_4 \subset L_1\} = W^{r+1}(k-r,l+r)$$
which forgets $L_2$.  $\pi$ is a $\p^r$ bundle and $\det(L_3/L_2)$ restricts to $\O_{\p^r}(1)$ on its fibers. Also
\begin{eqnarray*}
\wedge^s (L_3/L_2 \otimes (L_3/L_1)^\vee \{-2\}) \otimes \det(L_2/L_1)^{l-k+r} \otimes \det(L_3/L_2)^{l-k-1} \{rl+l-1\} \\
\cong \wedge^s (L_3/L_1)^\vee \otimes \det(L_3/L_2)^{s-r-1} \otimes \det(L_3/L_1)^{l-k+r} \{rl+l-1-2s\}.
\end{eqnarray*}
For $s = 1, \dots, r$, $\pi_* \det(L_3/L_2)^{s-r-1} = 0$, while for $s=0$ we get 
\begin{eqnarray*}
& & \pi_*(\det(L_3/L_2)^{-r-1} \otimes \det(L_3/L_1)^{l-k+r} \{rl+l-1\}) \\
&\cong& \pi_*(\omega_{\pi} \otimes \det(L_3/L_1)^{l-k+r-1} \{rl+l-1\}) \\
&\cong& \O_{W^{r+1}(k-r,l+r)} \det(L_3/L_1)^{l-k+r-1}[-r]\{rl+l-1\} \\
&\cong& \sF^{(r+1)}(k-r,l+r)[-r]\{r\},
\end{eqnarray*}
where we used that $\omega_\pi = \det((L_2/L_1) \otimes (L_3/L_2)^\vee) \cong \det(L_2/L_1) \otimes \det(L_3/L_2)^{-r}$. 
For $s=r+1$ we get
\begin{eqnarray*}
& & \pi_*(\det(L_3/L_1)^\vee \otimes \det(L_3/L_1)^{l-k+r}[r+1]\{rl+l-1-2(r+1)\}) \\
&\cong& \O_{W^{r+1}(k-r,l+r)} \det(L_3/L_1)^{l-k+r-1}\{(r+1)(l-1)\}[r+1]\{-r-2\} \\
&\cong& \sF^{(r+1)}(k-r,l+r)[r+1]\{-r-2\}.
\end{eqnarray*} 

We conclude that the spectral sequence which computes $\H^*(\pi_* \sK)$ from the cohomology of $\H^*(\sK)$ degenerates at the $E_2$ term and hence
$$\H^*(\pi_* \sK) \cong \sF^{(r+1)}(k-r,l+r) \otimes_\C (\C[-r]\{r\} \oplus \C[r+1]\{-r-2\}).$$
The result follows. 
The analogous result for $\sE$s instead of $\sF$s follows by adjunction. 
\end{proof}

\subsection{The commutator relation}

In this section we categorify the commutator relation. Roughly speaking, this says that $\E \circ \F = \F \circ \E \oplus \id \otimes H^\star(\p^\l)$. 

\begin{Proposition}\label{prop:comm1}
If $k \ge l$ then 
\begin{equation*}
\sF(k-1,l+1) * \sE(k-1,l+1) \cong \sE(k,l) * \sF(k,l) \oplus \sP
\end{equation*}
where $\H^*(\sP) \cong \O_\Delta \otimes_\C H^\star(\p^{k-l-1})$ and $\Delta \not\subset \mbox{supp}(\sE(k,l) * \sF(k,l))$.

Similarly, if $ k \le l $ then 
\begin{equation*}
\sE(k,l) * \sF(k,l) \cong \sF(k-1,l+1) * \sE(k-1,l+1) \oplus \sP'
\end{equation*}
where $\H^*(\sP') \cong \O_\Delta \otimes_\C H^\star(\p^{l-k-1})$. Moreover $\Delta \not\subset \mbox{supp}(\sF(k-1,l+1) * \sE(k-1,l+1))$.
\end{Proposition}
\begin{proof}
We deal with the case $k \le l$ since the case $k \ge l$ is analogous. 

Proposition \ref{prop:train} implies that the kernel $\sE(k,l) * \sF(k,l)$ is isomorphic to the convolution of kernels corresponding to the following sequence of maps
\begin{equation*}
\xymatrix {
X(k,1,l-1)_2 \ar[r]^{i_1} \ar[d]_{q_1} &Y(k,1,l-1) &X(k,1,l-1)_1 \ar[l]_{i_2} \ar[dr]_{q_2} \\
Y(k,l) & & & Y(k+1,l-1) }
\end{equation*}
\begin{equation*}
\xymatrix {
& X(k,1,l-1)_1 \ar[dl]_{q_1'} \ar[r]^{i_1'} &Y(k,1,l-1) &X(k,1,l-1)_2 \ar[l]_{i_2'} \ar[d]_{q_2'} \\
Y(k+1,l-1) & & & Y(k,l) }.
\end{equation*}
Here the upper $Y(k,1,l-1)$ carries the line bundle $\det(L_2/L_1)^{l-k-1}\{l-1\}$ while the lower $Y(k,1,l-1)$ caries the line bundle $\det(L_3/L_2)^\vee \otimes \det(L_1/L_0) \{k\}$. 

As in the proof of Proposition \ref{prop:Ecomps1} we first construct the fibre products
\begin{eqnarray*}
B &:=& X(k,1,l-1)_1 \times_{Y(k+1,l-1)} X(k,1,l-1)_1 \\
&\cong& \{ L_0 \overset{k}{\underset{k}{\rightrightarrows}} \begin{matrix} L_1' \\ L_1 \end{matrix} \overset{1}{\underset{1}{\rightrightarrows}} L_2 \xrightarrow{l-1} L_3:  zL_2 \subset L_0 \}
\end{eqnarray*}
with respect to $q_2$ and $q_1'$, as well as the fibre products
\begin{eqnarray*}
A_1 &:=& X(k,1,l-1)_2 \times_{Y(k,1,l-1)} B \\
&\cong& \{L_0 \overset{k}{\underset{k}{\rightrightarrows}} \begin{matrix} L_1' \\ L_1 \end{matrix} \overset{1}{\underset{1}{\rightrightarrows}} L_2 \xrightarrow{l-1} L_3: zL_2 \subset L_0 \text{ and } zL_3 \subset L_1\}
\end{eqnarray*}
and
\begin{eqnarray*}
A_2 &:=& B \times_{Y_{k,1,l-1}} X(k,1,l-1)_2 \\
&\cong& \{L_0 \overset{k}{\underset{k}{\rightrightarrows}} \begin{matrix} L_1' \\ L_1 \end{matrix} \overset{1}{\underset{1}{\rightrightarrows}} L_2 \xrightarrow{l-1} L_3: zL_2 \subset L_0 \text{ and } zL_3 \subset L_1'\}.
\end{eqnarray*}
Since $X(k,1,l-1)_1$ and $X(k,1,l-1)_2$ intersect in the expected dimension in $Y(k,1,l-1)$ and $q_2, q_1'$ are flat, we can use Lemma \ref{lem:fibreproduct} to replace the sequence of maps above by the following sequence:
\begin{equation*}
\xymatrix{
A_1 \ar[r]^{j_1} \ar[d]_{\tilde{q}_1} &B &A_2 \ar[l]_{j_2} \ar[d]_{\tilde{q}_2} \\
X(k,1,l-1)_2 \ar[d]_{q_1} & & X(k,1,l-1)_2 \ar[d]_{q_2'} \\
Y(k,l) & & Y(k,l).
}
\end{equation*}
Here $B$ carries the product of the pullback of the line bundles from the two copies of $Y(k,1,l-1)$, which is 
$\det(L_2/L_1)^{l-k-1} \otimes \det(L_3/L_2)^\vee \otimes \det(L_1'/L_0) \{l+k-1\}$. 

Unfortunately, the intersection $A_1 \cap A_2 \subset B$ contains two components of different dimensions. This means $A_1 \cap A_2$ cannot be an intersection of the right dimension in {\em any} subvariety of $B$. So we introduce two new spaces:
$$C := \{ L_0 \overset{k}{\underset{k}{\rightrightarrows}} \begin{matrix} L_1' \\ L_1 \end{matrix} \overset{1}{\underset{1}{\rightrightarrows}} L_2 \xrightarrow{l-k-1} L_3 \xrightarrow{k} L_4: zL_2 \subset L_0 \text{ and } zL_4 \subset L_2\}$$
and 
$$A_1' := \{L_0 \overset{k}{\underset{k}{\rightrightarrows}} \begin{matrix} L_1' \\ L_1 \end{matrix} \overset{1}{\underset{1}{\rightrightarrows}} L_2 \xrightarrow{l-k-1} L_3 \xrightarrow{k} L_4: zL_3 \subset L_0 \text{ and } zL_4 \subset L_1 \} \subset C.$$
There are natural maps $\pi_1:A_1' \rightarrow A_1$ and $\pi: C \rightarrow B$ obtained by forgetting $L_3$; but note that $A_1'$ is {\em not} the fibre product $A_1 \times_B C$. Instead, we have the following commutative diagram:
\begin{equation*}
\xymatrix {
A_1' \ar[r]^{j_1'} \ar[d]_{\pi_1} & C \ar[d]_{\pi} & C \times_B A_2 \ar[l]_{j_2'} \ar[d]_{\pi_2} \\
A_1 \ar[r]^{j_1} & B & A_2 \ar[l]_{j_2} }.
\end{equation*}
The middle vertical map $\pi$ is a $\G(l-k-1,l-1)$ bundle (hence flat) while the left vertical map $\pi_1$ is a birational map between smooth varieties (hence ${\pi_1}_* \circ \pi_1^* \cong \mbox{id}$). Consequently
$$j_2^* \circ {j_1}_* \cong j_2^* \circ ({j_1}_* \circ {\pi_1}_*) \circ \pi_1^* \cong j_2^* \circ \pi_* \circ {j_1'}_* \circ \pi_1^* \cong {\pi_2}_* \circ {j_2'}^* \circ {j_1'}_* \circ \pi_1^*.$$
Notice that if we also tensor by a line bundle on $B$ then the same equality holds after tensoring with the pullback of this line bundle to $C$. Thus we have the diagram 
\begin{equation*}
\xymatrix {
A_1' \ar[r]^{j_1'} \ar[d]_{\pi_1} & C & C \times_B A_2 \ar[l]_{j_2'} \ar[d]^{\pi_2} \\
A_1 \ar[d]_{\tilde{q_1}} & & A_2 \ar[d]^{\tilde{q_2}} \\
X(k,1,l-1)_2 \ar[d]_{q_1} & & X(k,1,l-1)_2 \ar[d]^{q_2} \\
Y(k,l) & & Y(k,l) }
\end{equation*}
where $C$ is equipped with the line bundle $\sL$, the pullback of the line bundle carried by $B$ 
\begin{eqnarray*}
\sL 
&=& \det(L_2/L_1)^{l-k-1} \det(L_4/L_2)^\vee \det(L_1'/L_0) \{l+k-1\} \\
&\cong& \det(L_2/L_1)^{l-k} \det(L_4/L_1)^\vee \det(L_1'/L_0) \{l+k-1\}. 
\end{eqnarray*}

The intersection $A_1' \cap (C \times_B A_2)$ inside $C$ consists of two pieces.  The first,
$$P_1 := \{L_0 \overset{k}{\underset{k}{\rightrightarrows}} \begin{matrix} L_1' \\ L_1 \end{matrix} \overset{1}{\underset{1}{\rightrightarrows}} L_2 \xrightarrow{l-k-1} L_3 \xrightarrow{k} L_4: zL_4 \subset L_1 = L_1' \text{ and } zL_3 \subset L_0 \}$$
is the locus where $L_1 = L_1'$, while the second
$$P_2 := \{L_0 \overset{k}{\underset{k}{\rightrightarrows}} \begin{matrix} L_1' \\ L_1 \end{matrix} \overset{1}{\underset{1}{\rightrightarrows}} L_2 \xrightarrow{l-k-1} L_3 \xrightarrow{k} L_4: zL_4 \subset L_1 \cap L_1' \text{ and } zL_3 \subset L_0 \text{ and } \dim \ker(z|_{L_4/L_0}) \ge l+1 \}$$
is the closure of the locus in $A_1' \cap (C \times_B A_2)$ where $L_1 \ne L_1'$. 

A dimension count shows that $\dim(P_1) = \dim(P_2) = \dim Y(k,l) + l-k-1$, while the codimensions of $A_1'$ and $C \times_B A_2$ inside $C$ are $(l-1)+k(l-k-1)$ and $l-1$ respectively. This means that the expected dimension of  $A_1' \cap (C \times_B A_2)$ inside $C$ is $\dim Y(k,l)$. 
A quick check shows that $A_1'$ and $C \times_B A_2$ are both contained inside 
$$D := \{L_0 \overset{k}{\underset{k}{\rightrightarrows}} \begin{matrix} L_1' \\ L_1 \end{matrix} \overset{1}{\underset{1}{\rightrightarrows}} L_2 \xrightarrow{l-k-1} L_3 \xrightarrow{k} L_4: zL_2 \subset L_0 \text{ and } zL_4 \subset L_2 \text{ and } zL_3 \subset L_1'\}$$
which has codimension $l-k-1$. Thus $A_1'$ and $C \times_B A_2$ intersect in the expected dimension inside $D$. 
Now 
\begin{center} (*) $D$ is carved out by the zero locus of $z: L_3/L_2 \rightarrow L_2/L_1' \{2\}$. \end{center} 
Thus we can apply Lemma \ref{lem:fibreproduct} to conclude that ${j_2'}^* ({j_1'}_* (\cdot) \otimes \sL)$ corresponds to a kernel $\sK$ whose cohomology is
$$\H^{-s}(\sK) \cong \left( \sL \otimes \wedge^s (L_3/L_2 \otimes (L_2/L_1')^\vee \{-2\} \right)|_{P_1 \cup P_2}.$$

To summarize, we have shown that the functor $\sE(k,l) * \sF(k,l)$ is given by the FM transform $p_{2*}(p_1^*(\cdot) \otimes \sK)$, where $p_1$ and $p_2$ are the two projections from $P_1 \cup P_2$ to $Y(k,l)$ given by forgetting $L_1',L_2,L_3$ and $L_1,L_2,L_3$ respectively. 

What remains is to compute $p_*(\sK)$ where 
$$p: P_1 \cup P_2 \rightarrow Y(k,l) \times Y(k,l)$$ 
is the map which forgets both $L_2$ and $L_3$. To do this we use the exact sequence for $\O_{P_1 \cup P_2}$
$$0 \rightarrow \O_{P_1 \cup P_2} \rightarrow \O_{P_1} \oplus \O_{P_2} \rightarrow \O_E \rightarrow 0$$
where $E = P_1 \cap P_2$ is a divisor in both $P_1$ and $P_2$. 
This allows us to calculate $$p_* \left( \sL \otimes \wedge^s (L_3/L_2 \otimes (L_2/L_1')^\vee \{-2\}) \right)$$ in a few distinct steps by studying the restriction of the this bundle to $E$, $P_1$ and $P_2$ separately. 

{\bf Step 1.} We first show that for all $s$ 
\begin{equation}\label{eq:pushE}
p_*(\sL \otimes \wedge^s (L_3/L_2 \otimes (L_2/L_1')^\vee \{-2\})|_E) = 0.
\end{equation}
Notice that
$$E \cong \{L_0 \xrightarrow{k} L_1 \xrightarrow{1} L_2 \xrightarrow{l-k-1} L_3 \xrightarrow{k} L_4: zL_3 \subset L_0 \text{,  } zL_4 \subset L_1 \text{, and } \dim \ker(z|_{L_4/L_0}) \ge l+1\}.$$
The map which forgets $L_2$ is a $\p^{l-k-1}$ bundle $f_1: E \rightarrow E_1,$ with 
$$E_1 = \{L_0 \xrightarrow{k} L_1 \xrightarrow{l-k} L_3 \xrightarrow{k} L_4: zL_3 \subset L_0 \text{ and } zL_4 \subset L_1 \text{ and } \dim \ker(z|_{L_4/L_0}) \ge l+1\}.$$
Now the relative cotangent bundle $\Omega_{f_1}$ is $(L_2/L_1) \otimes (L_3/L_2)^\vee $. Thus:
\begin{eqnarray*}
& & (\det(L_2/L_1)^{l-k-1} \otimes \det(L_3/L_2)^\vee) \otimes \wedge^s (L_3/L_2 \otimes (L_2/L_1)^\vee \{-2\}) \\
&\cong& \wedge^{l-k-1-s} ((L_2/L_1) \otimes (L_3/L_2)^\vee) \{-2s\} \\
&\cong& \Omega^{l-k-1-s}_{f_1} \{-2s\}
\end{eqnarray*}
so that
$$f_{1*}( \sL \otimes \wedge^s (L_3/L_2 \otimes (L_2/L_1')^\vee \{-2\})|_E) \cong \det(L_4/L_3)^\vee \otimes \det(L_1/L_0)[-l+k+1+s]\{l+k-1-2s\} |_{E_1}.$$

So to compute $p_* = f_{3*}f_{1*}$, we must compute the pushforward by the map $f_3: E_1 \rightarrow p(E)$ which forgets $L_3$.  This map is not flat, so we first resolve the singularities of $E_1$ using 
$$E_2 := \{L_0 \xrightarrow{k} L_1 \xrightarrow{l-k} L_3 \xrightarrow{1} L_{3 \frac{1}{2}} \xrightarrow{k-1} L_4: zL_{3\frac{1}{2}} \subset L_0 \text{ and } zL_4 \subset L_1 \}.$$
This gives a commutative diagram 
\begin{equation*}
\xymatrix{
& E_2 \ar[dr]^{f'_2} \ar[dl]_{f_2} & \\
E_1 \ar[dr]_{f_3} & & E'_1 \ar[dl]^{f'_3} \\
& p(E) &
}
\end{equation*}
where $f'_2$ forgets $L_3$ and $f'_3$ forgets $L_{3\frac{1}{2}}$. 

Now $E_2$ is an iterated Grassmannian bundle, hence smooth. Also, the map $f_2: E_2 \rightarrow E_1$ given by forgetting $L_{3 \frac{1}{2}}$ is generically one-to-one and surjective, because for a general point in $E_1$, 
$\dim(\ker(z|_{L_4/L_0})) = l+1$. This shows that $f_2$ is a resolution of $E_1$. By Lemma \ref{lem:ratsings} below, $E_1$ has rational singularities, which implies that that $f_{2*} \circ f_2^* = \id$. Thus we have
$$f_{3*} = f_{3*} \circ f_{2*} \circ f_2^* = f'_{3*} \circ f'_{2*} \circ f_2^*.$$

Finally, the map $f'_2$ is a flat $\p^{l-k}$ bundle and the restriction of $f_2^*(\det(L_4/L_3)^\vee \otimes \det(L_1/L_0))$ to a fibre of $f'_2$ is $\O(-1)$. Hence its pushforward via $f'_{2*}$ vanishes and we have proven \ref{eq:pushE}. 

{\bf Step 2.}
Next we show that 
\begin{equation}\label{eq:pushP1} 
p_* \left( (\sL \otimes \wedge^s (L_3/L_2 \otimes (L_2/L_1')^\vee \{-2\}))|_{P_1} \right) \cong \O_{\Delta}[s-(l-k-1)]\{-2s+(l-k-1)\} \in D(Y(k,l) \times Y(k,l)).
\end{equation}
On $P_1$ (where $L_1=L_1'$) we have 
\begin{eqnarray*}
& & \sL \otimes \wedge^s(L_3/L_2 \otimes (L_2/L_1')^\vee \{-2\})) \\
&\cong& \det(L_2/L_1)^{l-k}\det(L_4/L_1)^\vee \det(L_1/L_0) \{l+k-1\} \otimes \wedge^s(L_3/L_2 \otimes (L_2/L_1')^\vee \{-2\})) \\
&\cong& \det(L_4/L_3)^\vee \det(L_1/L_0) \det \left( (L_3/L_2)^\vee \otimes (L_2/L_1) \right) \otimes \wedge^s(L_3/L_2 \otimes (L_2/L_1)^\vee)) \{-2s+l+k-1\} \\
&\cong& \det(L_4/L_3)^\vee \det(L_1/L_0) \otimes \wedge^{l-k-1-s} \left( (L_3/L_2)^\vee \otimes (L_2/L_1) \right) \{-2s+l+k-1\}.
\end{eqnarray*}
Forgetting $L_2$ gives a $\p^{l-k-1}$ fibration $P_1 \rightarrow P_1'$,
$$P_1' := \{L_0 \xrightarrow{k} L_1 \xrightarrow{l-k} L_3 \xrightarrow{k} L_4: zL_3 \subset L_0 \text{ and } zL_4 \subset L_1 \},$$
with relative cotangent bundle $(L_3/L_2)^\vee \otimes (L_2/L_1)$. So pushing forward 
$$\det(L_4/L_3)^\vee \det(L_1/L_0) \otimes \wedge^{l-k-1-s} \left( (L_3/L_2)^\vee \otimes (L_2/L_1) \right) \{-2s+l+k-1\}$$
via the map which forgets $L_2$ gives the line bundle
$$\left( \det(L_4/L_3)^\vee \otimes \det(L_1/L_0) \right)_{P_1'}[s-(l-k-1)]\{-2s+l+k-1\}$$
on $P_1'$.

Now the map $f_3: P_1' \rightarrow \Delta \subset Y(k,l)\times Y(k,l)$ given by forgetting $L_3$ is generically one-to-one since there is no condition on the dimension of $\ker(z|_{L_4/L_0})$. The exceptional divisor of $f_3$ is the locus where $\dim \ker(z|_{L_4/L_0}) \ge l+1$, which is precisely $E_1$ from Step 1. Thus $E_1$ is carved out as the degeneracy locus of $z: L_4/L_3 \rightarrow L_1/L_0 \{2\}$, and
\begin{eqnarray*}
\O_{P_1'}(E_1) 
&\cong& \det((L_4/L_3)^\vee \otimes \det(L_1/L_0) \{2\}) \\
&\cong& \det((L_4/L_3)^\vee \otimes \det(L_1/L_0)) \{2k\}.
\end{eqnarray*} 
Since $f_3$ is a birational map onto the smooth variety $\Delta \cong Y(k,l)$, 
$$ f_{3*}\O_{P_1'} \cong \O_\Delta \in D(Y(k,l) \times Y(k,l)).$$ 
Since $E_1 \subset P_1'$ is cut out by the section $\det(z): \det(L_4/L_3) \rightarrow \det(L_1/L_0) \{2k\}$, we have 
$$\O_{E_1}(E_1) \cong ((L_4/L_3)^\vee \otimes \det(L_1/L_0) \{2k\})|_{E_1}.$$
As in Step 1, the pushforward $f_{3*}\O_{E_1}(E_1)=0$. Thus, from the exact triangle
$$\O_{P_1'} \rightarrow \O_{P_1'}(E_1) \rightarrow \O_{E_1}(E_1)$$
we find that $f_{3*}\O_{P_1'}(E_1)\cong \O_\Delta$. 
Hence the pushforward of 
$$\left( \det(L_4/L_3)^\vee \otimes \det(L_1/L_0) \{2k\} \right)_{P_1'}[s-(l-k-1)]\{-2s+l-k-1\}$$
is $\O_\Delta [s-(l-k-1)] \{-2s+(l-k-1)\}$ which proves (\ref{eq:pushP1}).  

{\bf Step 3.}
Finally, we show that 
\begin{equation}\label{eq:pushP2}
p_* \left( (\sL \otimes \wedge^s (L_3/L_2 \otimes (L_2/L_1')^\vee \{-2\})|_{P_2} \right) \cong 
\left\{
\begin{array}{ll}
\sF(k-1,l+1) * \sE(k-1,l+1) & \mbox{ if $s = 0$ } \\
0 & \mbox{ if $s \ne 0$ } 
\end{array}
\right.
\end{equation}
By Lemma \ref{lem:ratsings}, $P_2$ has rational singularities. So, just as in Step 1, we can first pull back to the natural resolution $Q_2$ of $P_2$
$$Q_2 = \{L_0 \xrightarrow{k-1} L_{\frac{1}{2}} \overset{1}{\underset{1}{\rightrightarrows}} \begin{matrix} L_1' \\ L_1 \end{matrix} \overset{1}{\underset{1}{\rightrightarrows}} L_2 \xrightarrow{l-k-1} L_3 \xrightarrow{1} L_{3\frac{1}{2}} \xrightarrow{k-1} L_4: zL_4 \subset L_{\frac{1}{2}} \text{ and } zL_{3\frac{1}{2}} \subset L_0 \}$$
and then push forward from there. 

The first map $g:Q_2 \rightarrow \tilde{P_2}$ forgets $L_3$ and projects onto
$$\tilde{P_2} := \{L_0 \xrightarrow{k-1} L_{\frac{1}{2}} \overset{1}{\underset{1}{\rightrightarrows}} \begin{matrix} L_1' \\ L_1 \end{matrix} \overset{1}{\underset{1}{\rightrightarrows}} L_2 \xrightarrow{l-k} L_{3\frac{1}{2}} \xrightarrow{k-1} L_4: zL_4 \subset L_{\frac{1}{2}} \text{ and } zL_{3\frac{1}{2}} \subset L_0 \}.$$
Notice that $g$ is a $\G(L_3/L_2, L_{3\frac{1}{2}}/L_2)=\p^{l-k-1}$ fibration. The restriction of $\sL$ to any fibre of $g$ is trivial, so $\sL \otimes \wedge^s (L_3/L_2 \otimes (L_2/L_1')^\vee)$ restricts to $\wedge^s (L_3/L_2)$. Since $L_3/L_2$ is the tautological bundle on the Grassmannian, the pushforward 
$g_* \left( (\sL \otimes \wedge^s (L_3/L_2 \otimes (L_2/L_1')^\vee \{-2\})|_{P_2} \right)$
vanishes for $s \ne 0$. When $s=0$, this pushforward gives
$$\sL := \det(L_2/L_1)^{l-k} \otimes \det(L_4/L_1)^\vee \otimes \det(L_1'/L_0) \{l+k-1\}.$$

We now push forward from $\tilde{P_2}$ to 
$$\tilde{P_2}' := \{L_0 \xrightarrow{k-1} L_{\frac{1}{2}} \overset{1}{\underset{1}{\rightrightarrows}} \begin{matrix} L_1' \\ L_1 \end{matrix} \overset{l-k+1}{\underset{l-k+1}{\rightrightarrows}} L_{3\frac{1}{2}} \xrightarrow{k-1} L_4: zL_4 \subset L_{\frac{1}{2}} \text{ and } zL_{3\frac{1}{2}} \subset L_0 \}$$
via the map $f:\tilde{P_2}\rightarrow \tilde{P_2}'$ which forgets $L_2$. Notice that $\tilde{P_2}'$ and $\tilde{P_2}$ are both smooth and that $f: \tilde{P_2} \rightarrow \tilde{P_2}'$ is a blowup along the locus where $L_1 = L_1'$. The exceptional divisor $E_f$ of $f$ is cut out as the zero locus of the natural inclusion map $L_1'/L_{\frac{1}{2}} \rightarrow L_2/L_1$ so that $\O_{\tilde{P_2}}(E_f) \cong (L_1/L_{\frac{1}{2}})^\vee \otimes L_2/L_1$. 

On the other hand, since the fibres of $f|_{E_f}$ are $\p^{l-k}$s, $f_*(\O_{\tilde{P_2}}(E_f^{\otimes i})) \cong \O_{\tilde{P_2}'}$ for $i=0, \dots, l-k$. Thus
\begin{eqnarray*}
f_*(\sL|_{\tilde{P_2}}) 
&\cong& f_*(\det(L_2/L_1)^{l-k} \otimes \det(L_4/L_1)^\vee \otimes \det(L_1'/L_0) \{l+k-1\}) \\
&\cong& f_*(E_f^{\otimes l-k} \otimes (L_1/L_{\frac{1}{2}})^{l-k} \otimes \det(L_4/L_1)^\vee \otimes \det(L_1'/L_0)) \{l+k-1\} \\
&\cong& \left( (L_1/L_{\frac{1}{2}})^{l-k} \otimes \det(L_4/L_1)^\vee \otimes \det(L_1'/L_0) \right)|_{\tilde{P_2}'} \{l+k-1\}
\end{eqnarray*}
where we use the projection formula to obtain the last line. Consequently, the left side of (\ref{eq:pushP2}) when $s=0$ is equal to 
\begin{equation}\label{eq:pushP2'}
f'_*((L_1/L_{\frac{1}{2}})^{l-k} \otimes \det(L_4/L_1)^\vee \otimes \det(L_1'/L_0)\{l+k-1\})
\end{equation}
where $f': \tilde{P_2'} \rightarrow Y(k,l) \times Y(k,l)$ forgets $L_{\frac{1}{2}}$ and $L_{3\frac{1}{2}}$. 

Finally, we need to calculate $\sF(k-1,l+1) * \sE(k-1,l+1)$. This can be done directly. Namely, the intersection 
$$W := \pi_{12}^{-1} W(k-1,l+1) \cap \pi_{23}^{-1} W(k-1,l+1) \subset Y(k,l) \times Y(k-1,l+1) \times Y(k,l)$$
equals 
$$\{L_0 \xrightarrow{k-1} L_1' \overset{1}{\underset{1}{\rightrightarrows}} \begin{matrix} L_1 \\ L_1'' \end{matrix} \overset{l}{\underset{l}{\rightrightarrows}} L_2: zL_1 \subset L_0 \text{ and } zL_1'' \subset L_0 \text{ and } zL_2 \subset L_1'\}$$
which is of the expected dimension. Thus 
\begin{eqnarray*}
\sF(k-1,l+1) * \sE(k-1,l+1) &\cong& \pi_{13*}(\pi_{12}^* (\O_{W(k-1,l+1)} \otimes \det(L_2/L_1')^\vee \det(L_1/L_0) \{k-1\}) \\
& & \otimes \pi_{23}^*(\O_{W(k-1,l+1)} \otimes (L_1'/L_1)^{l-k+1} \{l\})) \\
&\cong& \pi_{13*}(\O_W \otimes \det(L_2/L_1')^\vee \det(L_1/L_0) \otimes (L_1''/L_1')^{l-k+1} \{l+k-1\}) \\
&\cong& \pi_{13*}(\O_W \otimes \det(L_2/L_1'')^\vee \det(L_1/L_0) \otimes (L_1''/L_1')^{l-k} \{l+k-1\}). 
\end{eqnarray*}
By Lemma \ref{lem:ratsings}, $W$ has rational singularities, so we can first pull back from $W$ to 
$$\tilde{P_2}' \cong \{L_0 \xrightarrow{k-1} L_1' \overset{1}{\underset{1}{\rightrightarrows}} \begin{matrix} L_1 \\ L_1'' \end{matrix} \overset{l-k+1}{\underset{l-k+1}{\rightrightarrows}} L_{1\frac{1}{2}} \xrightarrow{k-1} L_2: zL_{1\frac{1}{2}} \subset L_0 \text{ and } zL_2 \subset L_1'\}.$$
Then $\sE(k-1,l+1) * \sF(k-1,l+1)$ is equal to  
$$f'_*((L_1''/L_1')^{l-k} \otimes \det(L_2/L_1'')^\vee \otimes \det(L_1/L_0) \{l+k-1\})$$
where $f'$ is the map from $\tilde{P_2}'$ which forgets $L_1'$ and $L_{1\frac{1}{2}}$. 
After relabeling  $L_1' \leftrightarrow L_{\frac{1}{2}}$, $L_1 \leftrightarrow L_1'$, $L_1'' \leftrightarrow L_1$, $L_{1\frac{1}{2}} \leftrightarrow L_{3\frac{1}{2}}$ and $L_2 \leftrightarrow L_4$ we find that this is the same as (\ref{eq:pushP2'}).  This completes the step.

Notice that  $f'(P_2')$ does not contain the diagonal, so that we also have $\Delta \not\subset \mbox{supp}(\sF(k-1,l+1) * \sE(k-1,l+1))$ (this was one of the conditions required in the definition of a geometric categorical
$\sl_2$ action.)

{\bf Step 4.} We now put together the calculations from Steps 1, 2 and 3 to compute $p_*(\sK)$. Consider the exact triangle $\sK' \rightarrow \sK \rightarrow \H^0(\sK \otimes \sL) \cong \sL$ obtained by projecting onto the top cohomology of $\sK$ (recall that we only know the cohomology of $\sK$ and not $\sK$ itself). The cohomology of $\sK'$ is 
$$\H^{-s}(\sK') \cong \left( \sL \otimes \wedge^s (L_3/L_2 \otimes (L_2/L'_1)^\vee \{-2\} \right)|_{P_1 \cup P_2}$$
where $s \ge 1$. From the calculations above we see that for $s \ge 1$
$$p_*(\sL \otimes \wedge^s (L_3/L_2 \otimes (L_2/L'_1)^\vee \{-2\})) \cong \O_\Delta [s-(l-k-1)] \{-2s+(l-k-1)\}.$$ 
Therefore the spectral sequence which computes $p_*(\sK')$ degenerates at the $E_2$ term, and we obtain 
$$\H^*(\sK') \cong \O_\Delta[l-k-1]\{-l+k+1\} \oplus \dots \oplus \O_\Delta[-l+k-1]\{l-k+1\}.$$

Now $p_*(\sL) \cong \O_\Delta[-l+k+1]\{l-k-1\} \oplus \sF(k-1,l+1) * \sE(k-1,l+1)$ and
$$p_*(\sK) \cong \Cone(p_*(\sL)[-1] \rightarrow p_*(\sK')).$$
However, we also have that (omitting the $\{\cdot\}$ shifts)
\begin{eqnarray*}
& & \Hom(\sF(k-1,l+1) * \sE(k-1,l+1)[-1], \O_\Delta[j])  \\
&\cong& \Hom(\sE(k-1,l+1)[-1], \sF_R(k-1,l+1)[j]) \\
&\cong& \Hom(\sE(k-1,l+1), \sE(k-1,l+1)[k-l+j]),
\end{eqnarray*}
which is zero if $j < l-k$ (since all the $\sE$s and $\sF$s are sheaves). Thus 
$$\Hom(\sF(k-1,l+1) * \sE(k-1,l+1)[-1], p_*(\sK')) = 0$$
and hence
$$p_*(\sK) \cong \sF(k-1,l+1) * \sE(k-1,l+1) \oplus \sP'$$
where $\H^*(\sP') \cong \O_\Delta \otimes_\C H^\star(\p^{l-k-1})$.  This completes the proof.
\end{proof}

\begin{Lemma}\label{lem:ratsings} The varieties $E_1, P_2$ and $W$ in the proof of Proposition \ref{prop:comm1} have rational singularities. 
\end{Lemma}
\begin{proof}
A variety $X$ has rational singularities if for some (or equivalently any) resolution of singularities 
$\pi: Y \rightarrow X$ one has $\pi_*(\O_Y) = \O_X$. 

Given a resolution $\pi: Y \rightarrow X$ of singularities, to show that $\pi_*(\O_Y) = \O_X$ it suffices to show that 
$X$ is normal, Cohen-Macaulay, and that $R^0 \pi_* \omega_Y \cong \omega_X$ 
(see  Theorem 5.10 of \cite{KM}).
But if one can show that $R^0 \pi_* \omega_Y$ and $\omega_X$ are isomorphic up to tensoring by a line bundle then it follows that they are isomorphic, since they agree in codimension one (away from the indeterminacy locus of $X \dashrightarrow Y$).
Thus, to show that $X$ has rational singularities it suffices to show that $X$ is normal, Cohen-Macaulay, and that for some resolution of singularities, $\omega_X$ and $R^0 \pi_* \omega_Y$ are isomorphic up to tensoring by a line bundle. 

Begin first with the variety
$$E_1 = \{L_0 \xrightarrow{k} L_1 \xrightarrow{l-k} L_3 \xrightarrow{k} L_4: zL_3 \subset L_0 \text{ and } zL_4 \subset L_1 \text{ and } \dim \ker(z|_{L_4/L_0}) \ge l+1\}$$
$E_1$ is a codimension one subvariety of the smooth variety
$$\{L_0 \xrightarrow{k} L_1 \xrightarrow{l-k} L_3 \xrightarrow{k} L_4: zL_3 \subset L_0 \text{ and } zL_4 \subset L_1 \}$$
($E_1$ is cut out by $\det(z): \det(L_4/L_3) \rightarrow \det(L_1/L_0) \{2k\}$). Thus $E_1$ is a local complete intersection (and hence Cohen-Macaulay). Since $E_1$ is smooth in codimension one it is also normal. 

Now $E_1$ has a resolution of singularities given by
$$E_2 = \{L_0 \xrightarrow{k} L_1 \xrightarrow{l-k} L_3 \xrightarrow{1} L_{3 \frac{1}{2}} \xrightarrow{k-1} L_4: zL_{3 \frac{1}{2}} \subset L_0 \text{ and } zL_4 \subset L_1 \}.$$
The dualizing sheaf of $E_2$ is
\begin{eqnarray*}
\omega_{E_2} &\cong&
[\det(L_3/L_1) \det(L_{3\frac{1}{2}}/L_3)^{-(l-k)}] \otimes [\det(L_4/L_{3\frac{1}{2}})^{m-l} \det(z^{-1}L_1/L_4)^{-(k-1)})] \otimes \\
& & [\det(L_1/L_0)^{l-k+1} \det(L_{3\frac{1}{2}}/L_1)^{-k})] \otimes [\det(L_{3\frac{1}{2}})^{m-l-1} \det(z^{-1}L_0/L_{3\frac{1}{2}})^{-l-1}]
\end{eqnarray*}
where we describe $E_2$ as an iterated fibration by first forgetting $L_3$, then $L_4$, $L_1$ and $L_{3\frac{1}{2}}$, in that order. We do not need to simplify this expression completely-
if we factor out all copies of $\det(L_{3\frac{1}{2}})$ so that the remaining factors do not contain any expression involving $L_{3\frac{1}{2}}$, we find that we have collected
$$-(l-k) -(m-l) + (-k) + (m-l-1) + (l+1) = 0$$
copies of $\det(L_{3\frac{1}{2}})$. This means that $\omega_{E_2}$ is the pullback of a line bundle on $E_1$. So, by the projection formula, the pushforward of $\omega_{E_2}$ is a line bundle on $E_1$.  Thus
$\omega_{E_1}$ and $R^0 \pi_* \omega_{E_2}$ agree up to tensoring with a line bundle, as desired.

Next we consider
$$P_2 = \{L_0 \overset{k}{\underset{k}{\rightrightarrows}} \begin{matrix} L_1' \\ L_1 \end{matrix} \overset{1}{\underset{1}{\rightrightarrows}} L_2 \xrightarrow{l-k-1} L_3 \xrightarrow{k} L_4: zL_4 \subset L_1 \cap L_1' \text{ and } zL_3 \subset L_0 \text{ and } \dim \ker(z|_{L_4/L_0}) \ge l+1 \}.$$
Recall that $P_1$ and $P_2$ are the irreducible components of 
$$P= P_1\cup P_2 = \{L_0 \overset{k}{\underset{k}{\rightrightarrows}} \begin{matrix} L_1' \\ L_1 \end{matrix} \overset{1}{\underset{1}{\rightrightarrows}} L_2 \xrightarrow{l-k-1} L_3 \xrightarrow{k} L_4: zL_4 \subset L_1, zL_4 \subset L_1' \text{ and } zL_3 \subset L_0 \}$$
and that 
$$P_1 = \{L_0 \xrightarrow{k} L_1 \xrightarrow{1} L_2 \xrightarrow{l-k-1} L_3 \xrightarrow{k} L_4: zL_4 \subset L_1 \text{ and } zL_3 \subset L_0 \}$$
is smooth. Now $P$ is naturally a codimension $k$ subvariety of the Grassmannian bundle $\G(k,L_2/L_0) \rightarrow P_1$. Explicitly, it is cut out by the section of the map 
$z:L_4/L_3 \rightarrow L_2/L_1' \{2\}$, where $L_1'$ is the tautological bundle of the Grassmannian bundle $\G(k,L_2/L_0)$. It follows that $P_1 \cup P_2$ is a local complete intersection. Now the intersection $P_1 \cap P_2$ is isomorphic to 
$$\{L_0 \xrightarrow{k} L_1 \xrightarrow{1} L_2 \xrightarrow{l-k-1} L_3 \xrightarrow{k} L_4: zL_4 \subset L_1, zL_3 \subset L_0 \text{ and } \dim \ker(z|_{L_4/L_0}) \ge l+1 \}$$
and has codimension one inside $P_1$, since it is carved out as the degeneracy locus of $z: L_4/L_3 \rightarrow L_1/L_0 \{2\}$.   Moreover, the locus where $P_2$ is singular is contained in the locus where $\dim \ker(z|_{L_4/L_0}) \ge l+2$; this locus has codimension at least $2$. Thus by Lemma \ref{lem:CM}, we find that $P_2$ is normal and Cohen-Macaulay. 
 
Now $P_2$ has a natural resolution $\pi: Q \rightarrow P_2$,
$$Q := \{L_0 \xrightarrow{k-1} L_{\frac{1}{2}} \overset{1}{\underset{1}{\rightrightarrows}} \begin{matrix} L_1' \\ L_1 \end{matrix} \overset{1}{\underset{1}{\rightrightarrows}} L_2 \xrightarrow{l-k-1} L_3 \xrightarrow{k} L_4: zL_4 \subset L_{\frac{1}{2}} \text{ and } zL_{3} \subset L_0 \},$$
where the map $\pi$ forgets $L_{\frac{1}{2}}$.
We claim that $\omega_{P_2}$ and $R^0 \pi_*(\omega_Q)$ agree up to tensoring with a line bundle. 
The canonical bundle of $Q$ is
\begin{eqnarray*}
\omega_{Q} 
&\cong& (\det(L_1/L_{\frac{1}{2}}) \det(L_2/L_1)^{-1}) (\det(L_1'/L_{\frac{1}{2}}) \det(L_2/L_1')^{-1}) (\det(L_2/L_{\frac{1}{2}})^{l-k-1} \det(L_3/L_2)^{-2}) \\
& & (\det(L_4/L_3)^{m-l-1} \det(z^{-1}L_{\frac{1}{2}}/L_4)^{-k}) (\det(L_{1/2}/L_0)^{l-k+1} \det(L_3/L_{\frac{1}{2}})^{-k+1}) \\
& & (\det(L_3/L_0)^{m-l} \det(z^{-1}L_0/L_3)^{-l})
\end{eqnarray*}
obtained by forgetting $L_1, L_1', L_2, L_4, L_{\frac{1}{2}}$ and $L_3$ in that order. We collect all the factors involving $\det(L_\frac{1}{2}/L_0)$, finding
$$-1-1-(l-k-1)-k+(l-k+1)-(-k+1) = -1$$
such factors. So, by the projection formula, 
$$R^0 \pi_*(\omega_{Q})\simeq R^0 \pi_*(\det(L_{\frac{1}{2}}/L_0)^{-1}) \otimes \sL$$ 
for some line bundle $\sL$.

Now consider the zero locus of the map $L_1'/L_{\frac{1}{2}} \rightarrow L_2/L_1$ inside $Q$. This is the divisor $D$ where $L_1=L_1'$ ({\it i.e.} the intersection $P_1 \cap P_2 \subset P_2$). The short exact sequence
$$\O_Q(-D) \rightarrow \O_Q \rightarrow \O_D$$
produces a long exact sequence
$$R^0\pi_* \O_Q(-D) \rightarrow \O_{P_2} \rightarrow \O_{\pi(D)} \rightarrow \dots$$
where we use that $R^0\pi_*(\O_Q) = \O_{P_2}$ ($P_2$ is normal) and $R^0 \pi_* \O_D = \O_{\pi(D)}$ ($\pi(D)$ is also normal). Thus $R^0 \pi_* \O_Q(-D) \cong \O_{P_2}(-\pi(D))$ since $\pi(D) \subset P_2$ is the divisor in $P_2$ where $L_1 = L_1'$. Since 
$$\O_Q(-D) = (L_2/L_1)^{-1} \otimes (L_1'/L_{\frac{1}{2}}) \cong (L_2/L_1)^{-1} \det(L_1'/L_0) \det(L_{\frac{1}{2}}/L_0)^{-1}$$ 
this means that $R^0 \pi_*(\omega_{Q})$ is equal to $\O_{P_2}(-\pi(D))$, up to tensoring with a line bundle. Thus we need to show that $\O_{P_2}(-\pi(D))$ is equal to $\omega_{P_2}$ up to tensoring by a line bundle (notice that we do not know if $\O_{P_2}(-\pi(D))$). 

On the other hand, we can compute $\omega_{P_2}$ as follows. Consider the standard exact triangle
$$\O_{P_1}(-(P_1 \cap P_2)) \rightarrow \O_{P_1 \cup P_2} \rightarrow \O_{P_2}$$
inside the Grassmannian bundle $\G(k,L_2/L_0)$ over $P_1$ mentioned above.  Now $P_1$ is smooth (so $\O_{P_1}(-(P_1 \cap P_2))$ is locally free) and $P_1 \cup P_2$ is a local complete intersection (so $\omega_{P_1 \cup P_2}$ is a locally free sheaf). Dualizing the above triangle we get the exact triangle
$$\omega_{P_2} \rightarrow \omega_{P_1 \cup P_2} \rightarrow \omega_{P_1}(P_1 \cap P_2)$$
where $\omega_{P_2}$ is a sheaf (since $P_2$ is Cohen-Macaulay) and the right two terms are locally free sheaves. We may compare this exact triangle with the standard exact triangle
$$\O_{P_2}(-\pi(D)) \rightarrow \O_{P_1 \cup P_2} \rightarrow \O_{P_1}$$
 and find that $\omega_{P_2}$ is isomorphic to $\O_{P_2}(- \pi(D))$ up to tensoring by a line bundle. This means that $R^0 \pi_*(\omega_Q)$ and $\omega_{P_2}$ agree up to tensoring with a line bundle, as desired.
 
Finally, we show that
$$W = \{L_0 \xrightarrow{k-1} L_1' \overset{1}{\underset{1}{\rightrightarrows}} \begin{matrix} L_1 \\ L_1'' \end{matrix} \overset{l}{\underset{l}{\rightrightarrows}} L_2: zL_1 \subset L_0 \text{ and } zL_1'' \subset L_0 \text{ and } zL_2 \subset L_1'\}$$
has rational singularities.  Note that we have $k\le l$ in the above description of $W$.
$W$ is naturally a subvariety of the Grassmannian bundle 
$$\G(1,L_2/L_1') \rightarrow \{ L_0 \xrightarrow{k-1} L_1' \xrightarrow{1} L_1 \xrightarrow{l} L_2: zL_1 \subset L_0 \text{ and } zL_2 \subset L_1'\}.$$
The fibres of $W$ are generically $\G(1,(l+1)-(k-1))$, since $\ker z|_{L_2/L_0}$ on the base generically has dimension $l+1$. This means that $W \subset \G(1,L_2/L_1')$ has codimension $l-(l-k+1)=k-1$ and is cut out by the zero section of $z:L_1''/L_1 \rightarrow L_1'/L_0 \{2\}$, where $L_1''$ is the tautological bundle on $\G(1,L_2/L_1')$. Thus $W$ is a local complete intersection. The singular locus of $W$ is contained in the locus where 
$\dim (\ker z|_{L_2/L_0}) \ge l+2$, which is codimension $\ge 2$. Thus $W$ is Gorenstein and normal. 

Now a natural resolution of $W$ is 
$$W' = \{L_0 \xrightarrow{k-1} L_1' \overset{1}{\underset{1}{\rightrightarrows}} \begin{matrix} L_1 \\ L_1'' \end{matrix} \overset{1}{\underset{1}{\rightrightarrows}} L_{\frac{3}{2}} \xrightarrow{l-1} L_2: zL_{\frac{3}{2}} \subset L_0 \text{ and } zL_2 \subset L_1'\}$$
whose canonical bundle is
\begin{eqnarray*}
\omega_{W'} &\cong& (\det(L_1/L_1') \det(L_{\frac{3}{2}}/L_1)^{-1}) (\det(L_1''/L_1) \det(L_{\frac{3}{2}}/L_1'')^{-1}) \\
& & (\det(L_2/L_{\frac{3}{2}})^{m-l-1} \det(z^{-1}L_1'/L_2)^{-l+1}) (\det(L_1'/L_0)^{2} \det(L_{\frac{3}{2}}/L_1')^{-k+1}) \\
& & (\det(L_{\frac{3}{2}}/L_0)^{m-k-1} \det(z^{-1}L_0/L_{\frac{3}{2}})^{-k-1}.
\end{eqnarray*}
In order to obtain this description of $\omega_{W'}$, we forget $L_1,L_1'',L_2,L_1'$ and $L_{\frac{3}{2}}$ in that order. 
Collecting all factors involving $\det(L_{\frac{3}{2}}/L_0)$ we find 
$$-1-1-(m-l-1)+(-k+1)+(m-k-1)-(-k-1) = l-k \ge 0$$
such terms. 
Now the exceptional divisor of the forgetful map $\pi: W' \rightarrow W$ which forgets $L_{\frac{3}{2}}$ is the locus $D \subset W'$ where $L_1 = L_1''$. $D$ is cut out by the zero locus of the natural map $L_1/L_1' \rightarrow L_{\frac{3}{2}}/L_1''$. Thus $\O_{W'}(D) \cong \det(L_{\frac{3}{2}}/L_1'') \det(L_1/L_1')^{-1}$. But $R^0 \pi_*(\O_{W'}((l-k)(D))) \cong \O_W$ since $l-k \ge 0$. Thus 
\begin{eqnarray*}
R^0 \pi_*(\det(L_{\frac{3}{2}}/L_0)^{l-k}) 
&\cong& R^0 \pi_*(\O_{W'}((l-k)(D)) \otimes \det(L_1/L_1')^{l-k} \det(L_1''/L_0)^{l-k}) \\
&\cong& \det(L_1/L_1')^{l-k} \det(L_1''/L_0)^{l-k} 
\end{eqnarray*}
by the projection formula. Hence $R^0 \pi_*(\omega_{W'})$ is a line bundle, which completes the proof. 
\end{proof}

\begin{Lemma}\label{lem:CM} Suppose $X_1 \cup X_2 \subset X$ is a local complete intersection inside the smooth variety $X$. If $X_1$ is smooth, $X_1 \cap X_2 \subset X_1$ is a divisor and $X_2$ is smooth in codimension one then $X_2$ is normal and Cohen-Macaulay.
\end{Lemma}
\begin{proof}
By definition, $X_2$ is Cohen-Macaulay if its dualizing complex $\omega_{X_2}$ is a sheaf. Now consider the exact triangle
$$\O_{X_1}(-X_1 \cap X_2) \rightarrow \O_{X_1 \cup X_2} \rightarrow \O_{X_2}.$$
Dualizing we get the exact triangle
\begin{equation}\label{eq:ses}
\omega_{X_2} \rightarrow \omega_{X_1 \cup X_2} \xrightarrow{r} \omega_{X_1}(X_1 \cap X_2).
\end{equation}
The second term is a line bundle on $X_1 \cup X_2$ since $X_1 \cup X_2$ is a local complete intersection (and hence Gorenstein) while the third term is also a line bundle ($X_1$ is smooth so the divisor $X_1 \cap X_2 \subset X_1$ is Cartier). 

Now any map $r: \omega_{X_1 \cup X_2} \rightarrow \omega_{X_1}(X_1 \cap X_2)$ as in (\ref{eq:ses}) is the composition 
$$\omega_{X_1 \cup X_2} \rightarrow \omega_{X_1 \cup X_2}|_{X_1} \rightarrow \omega_{X_1}(X_1 \cap X_2)$$
where the first map is restriction. This means that $r$ is either surjective or the cokernel has codimension one since the restriction map is surjective and the second map is between two line bundles on the smooth variety $X_1$. But $\omega_{X_2}$ is a sheaf over the locus where $X_2$ is smooth. The complement of this locus has codimension at least two so this means that the support of $\H^1(\omega_{X_2})$ has codimension at least two. Thus $r$ must be surjective and $\omega_{X_2}$ must be a sheaf.

Finally, $X_2$ is normal since normality is the same as being $S_2$ (which is implied by being Cohen-Macaulay) and smooth in codimension one. 
\end{proof}

Condition (i) (for a geometric categorical $\sl_2$ action) follows since $Y(k,l)$ is proper. 

Finally, to see condition (vi) suppose $k \ge l$. Then the image of $\supp \sE^{(r)}(k,l)$ inside $Y(k,l)$ contains points where the restriction of $z$ to $L_2$ has kernel of dimension $k+r$. On the other hand, the image of $\supp \sE^{(r+s)}(k,l)$ inside $Y(k,l)$ is contained in the locus where the restriction of $z$ to $L_2$ has kernel of dimension $\ge k+r+s$. Thus the former cannot be contained in the latter. This proves the first part of condition (vi) -- the second part is proved similarly. 

As a corollary of this and the propositions from section \ref{sec:proofs}, we conclude the proof of Theorem \ref{thm:main}.

\begin{proof}[Proof of Theorem \ref{thm:main}]
The requirements for a geometric categorical $\C^\times$ equivariant $\sl_2$ action are given in section \ref{se:geomsl2def}, and are checked in Propositions \ref{prop:Eadj}, \ref{prop:Ecomps1}, \ref{prop:Ecomps2} and \ref{prop:comm1} and in the comments above. 
\end{proof}

\end{document}